\documentclass{article}%
\usepackage{amsfonts}
\usepackage{amsmath}
\usepackage{amssymb}
\usepackage{graphicx}%
\newtheorem{theorem}{Theorem}

\newtheorem{definition}[theorem]{Definition}

\newtheorem{lemma}[theorem]{Lemma}

\newtheorem{remark}[theorem]{Remark}

\newenvironment{proof}[1][Proof]{\noindent\textbf{#1.} }{\ \rule{0.5em}{0.5em}}
\begin{document}

\title{A general necessary and sufficient optimality conditions for singular control problems}
\author{Seid Bahlali\thanks{Laboratory of applied mathematics, University Med Khider,
P.O. Box 145, Biskra 07000, Alg\'{e}ria. \ sbahlali@yahoo.fr}}
\date{}
\maketitle

\begin{abstract}
We consider a stochastic control problem where the set of strict (classical)
controls is not necessarily convex and the the variable control has two
components, the first being absolutely continuous and the second singular. The
system is governed by a nonlinear stochastic differential equation, in which
the absolutely continuous component of the control enters both the drift and
the diffusion coefficients. By introducing a new approach, we establish
necessary and sufficient optimality conditions for two models. The first
concerns the relaxed-singular controls, who are a pair of processes whose
first component is a measure-valued processes. The second is a particular case
of the first and relates to strict-singular control problems. These results
are given in the form of global stochastic maximum principle by using only the
first order expansion and the associated adjoint equation. This improves and
generalizes all the previous works on the maximum principle of controlled
stochastic differential equations.

\ 


\ 


\end{abstract}

\section{Introduction}

\ 

We study a stochastic control problem where the system is governed by a
nonlinear stochastic differential equation (SDE\ for short) of the type
\[
\left\{
\begin{array}
[c]{l}%
dx_{t}^{\left(  v,\eta\right)  }=b\left(  t,x_{t}^{\left(  v,\eta\right)
},v_{t}\right)  dt+\sigma\left(  t,x_{t}^{\left(  v,\eta\right)  }%
,v_{t}\right)  dW_{t}+G_{t}d\eta_{t},\\
x_{0}^{\left(  v,\eta\right)  }=x,
\end{array}
\right.
\]
where $b,\sigma$ and $G$ are given functions, $x$ is the initial data and
$W=\left(  W_{t}\right)  _{t\geq0}$ is a standard Brownian motion, defined on
a filtered probability space $\left(  \Omega,\mathcal{F},\left(
\mathcal{F}_{t}\right)  _{t\geq0},\mathcal{P}\right)  ,$ satisfying the usual conditions.

The control variable, called strict-singular control, is a suitable process
$\left(  v,\eta\right)  $ where $v:\left[  0,T\right]  \times\Omega
\longrightarrow U_{1}\subset\mathbb{R}^{k}$, $\eta:\left[  0,T\right]
\times\Omega\longrightarrow U_{2}=\left(  [0,\infty)\right)  ^{m}$ are
$B\left[  0,T\right]  \otimes\mathcal{F}$-measurable, $\left(  \mathcal{F}%
_{t}\right)  $- adapted, and $\eta$ is an increasing process (componentwise),
continuous on the left with limits on the right with $\eta_{0}=0$. We denote
by $\mathcal{U}$ the class of all strict-singular controls.

The criteria to be minimized, over the set $\mathcal{U}$, has the form%
\[
J\left(  v,\eta\right)  =E\left[  g\left(  x_{T}^{\left(  v,\eta\right)
}\right)  +\int_{0}^{T}h\left(  t,x_{t}^{\left(  v,\eta\right)  }%
,v_{t}\right)  dt+\int_{0}^{T}k_{t}d\eta_{t}\right]  ,
\]
where, $g,\ h$ and $k$ are given maps and $x_{t}^{\left(  v,\eta\right)  }$ is
the trajectory of the system controlled by $\left(  v,\eta\right)  $.

A control $\left(  u,\xi\right)  \in\mathcal{U}$ is called optimal if it
satisfies%
\[
J\left(  u,\xi\right)  =\underset{\left(  v,\eta\right)  \in\mathcal{U}}{\inf
}J\left(  v,\eta\right)  .
\]

This kind of stochastic control problems have been studied by many authors,
both by the dynamic programming approach and by the Pontryagin stochastic
maximum principle. The first approach was studied by Ben\u{e}s, Shepp and
Witsenhausen $\left[  6\right]  ,$ Chow, Menaldi and Robin $\left[  10\right]
,$ Karatzas and Shreve $\left[  21\right]  ,$ Davis and Norman $\left[
11\right]  , $ Haussmann and Suo $\left[  17,18,19\right]  .$ See $\left[
17\right]  $ for a complete list of references on the subject. It was shown in
particular that the value function is solution of a variational inequality,
and the optimal state is a reflected diffusion at the free boundary. Note that
in $\left[  17\right]  ,$ the authors apply the compactification method to
show existence of an optimal singular control.

In this paper, we are concerned with the second approach, whose objective is
to establish necessary (as well as sufficient) conditions for optimality of
controls. The first version of the stochastic maximum principle that covers
singular control problems was obtained by Cadenillas and Haussmann $\left[
8\right]  $, in which they consider linear dynamics, convex cost criterion and
convex state constraints. The method used in $\left[  8\right]  $ is based on
the known principle of convex analysis, related to the minimization of convex,
continuous and G\^{a}teaux-differentiable functional defined on a convex
closed set. Necessary conditions of optimality for non linear SDEs with convex
control domain, where the coefficients depend explicitly on the absolutely
part of the control, was derived by Bahlali and Chala $\left[  1\right]  $ by
applying a convex perturbation on the pair of controls. The result in then
obtained in weak form. Bahlali and Mezerdi $\left[  2\right]  $ generalize the
work of $\left[  1\right]  $ to the case of nonconvex control domain, and
derive necessary optimality conditions by using a strong perturbation (spike
variation) on the absolutely continuous component of the control and a convex
perturbation on the singular one. The Peng stochastic maximum principle is
then used and te result is given with two adjoint equations and a variational
inequality of the second order. Version of stochastic maximum principle for
relaxed-singular controls was established by Bahlali, Djehiche and Mezerdi
$\left[  4\right]  $ in the case of uncontrolled diffusion, by using the
previous works on strict-singular controls, Ekeland's variational principle
and some stability properties of the trajectories and adjoint processes with
respect to the control variable.

In a recent work, Bahlali $\left[  5\right]  $ generalizes and improves all
the previous results on stochastic maximum principle for controlled SDEs, by
introducing a new approach and establish necessary and sufficient optilmality
conditions for both relaxed and strict controls, by using only the first order
expansion and the associated adjoint equation. The main idea of $\left[
5\right]  $, is to use the property of convexity of the set of relaxed
controls and treat the problem with the convex perturbation on relaxed
controls (instead of the spike variation on strict one).

Our aim in this paper, is to follow the method used by $\left[  5\right]  $
and derive necessary as well as sufficient conditions of optimality in the
form of global stochastic maximum principle, for both relaxed-singular and
strict-singular controls, without using the second order expansion. We
introduce then a bigger new class $\mathcal{R}$ of\ processes by replacing the
$U_{1}$-valued process $\left(  v_{t}\right)  $ by a $\mathbb{P}\left(
U_{1}\right)  $-valued process $\left(  q_{t}\right)  $, where $\mathbb{P}%
\left(  U_{1}\right)  $ is the space of probability measures on $U_{1}$
equipped with the topology of stable convergence. This new class of processes
is called relaxed-singular controls and have a richer structure, for which the
control problem becomes solvable.

In the relaxed-singular model, the system is governed by the SDE%
\[
\left\{
\begin{array}
[c]{l}%
dx_{t}^{\left(  q,\eta\right)  }=%
{\displaystyle\int\nolimits_{U_{1}}}
b\left(  t,x_{t}^{\left(  q,\eta\right)  },a\right)  q_{t}\left(  da\right)
dt+%
{\displaystyle\int\nolimits_{U_{1}}}
\sigma\left(  t,x_{t}^{\left(  q,\eta\right)  },a\right)  q_{t}\left(
da\right)  dW_{t}+G_{t}d\eta_{t},\\
x_{0}^{\left(  q,\eta\right)  }=x.
\end{array}
\right.
\]

The functional cost to be minimized, over the class $\mathcal{R}$ of
relaxed-singular controls, is defined by%
\[
\mathcal{J}\left(  q,\eta\right)  =\mathbb{E}\left[  g\left(  x_{T}^{\left(
q,\eta\right)  }\right)  +\int_{0}^{T}%
{\displaystyle\int\nolimits_{U_{1}}}
h\left(  t,x_{t}^{\left(  q,\eta\right)  },a\right)  q_{t}\left(  da\right)
dt+\int_{0}^{T}k_{t}d\eta_{t}\right]  .
\]

A relaxed-singular control $\left(  \mu,\xi\right)  $ is called optimal if it
solves
\[
\mathcal{J}\left(  \mu,\xi\right)  =\inf\limits_{\left(  q,\eta\right)
\in\mathcal{R}}\mathcal{J}\left(  q,\eta\right)  .
\]

The relaxed-singular control problem finds its interest in two essential
points. The first is that an optimal solution exists. Haussmann and Suo
$\left[  17\right]  $ have proved that the relaxed-singular control problem
admits an optimal solution under general conditions on the coefficients.
Indeed, by using a compactification method and under some mild continuity
hypotheses on the data, it is shown by purely probabilistic arguments that an
optimal solution for the problem exists. Moreover, the value function is shown
to be Borel measurable. The second interest is that it is a generalization of
the strict-singular control problem. Indeed, if $q_{t}\left(  da\right)
=\delta_{v_{t}}\left(  da\right)  $ is a Dirac measure concentrated at a
single point $v_{t}$ of $U_{1}$, then we get a strict-singular control problem
as a particular case of the relaxed one.

To achieve the objective of this paper and establish necessary and sufficient
optimality conditions for these two models, we proceed as follows.

Firstly, we give the optimality conditions for relaxed controls. The main idea
is to use the fact that the set of relaxed controls is convex. Then, we
establish necessary optimality conditions by using the classical way of the
convex perturbation method. More precisely, if we denote by $\left(  \mu
,\xi\right)  $ an optimal relaxed control and $\left(  q,\eta\right)  $ is an
arbitrary element of $\mathcal{R}$, then with a sufficiently small $\theta>0 $
and for each $t\in\left[  0,T\right]  $, we can define a perturbed control as
follows%
\[
\left(  \mu_{t}^{\theta},\xi_{t}^{\theta}\right)  =\left(  \mu_{t},\xi
_{t}\right)  +\theta\left[  \left(  q_{t},\eta_{t}\right)  -\left(  \mu
_{t},\xi_{t}\right)  \right]  .
\]

We derive the variational equation from the state equation, and the
variational inequality from the inequality%
\[
0\leq\mathcal{J}\left(  \mu^{\theta},\xi^{\theta}\right)  -\mathcal{J}\left(
\mu,\xi\right)  .
\]

By using the fact that the drift, the diffusion and the running cost
coefficients are linear with respect to the relaxed control variable,
necessary optimality conditions are obtained directly in the global form. The
result is given by using only the first-order expansion and the associated
adjoint equations

To enclose this part of the paper, we prove under minimal additional
hypotheses, that these necessary optimality conditions for relaxed-singular
controls are also sufficient.

The second main result in the paper characterizes the optimality for
strict-singular control processes. It is directly derived from the above
results by restricting from relaxed to strict-singular controls. The main idea
is to replace the relaxed controls by a Dirac measures charging a strict
controls. Thus, we reduce the set $\mathcal{R}$ of relaxed-singular controls
and we minimize the cost $\mathcal{J}$ over the subset $\delta\left(
\mathcal{U}_{1}\right)  \times\mathcal{U}_{2}=\left\{  \left(  q,\eta\right)
\in\mathcal{R}\text{ \ / }\ q=\delta_{v}\ \ ;\ \ v\in\mathcal{U}_{1}\right\}
$. Then, we derive necessary optimality conditions by using only the first
order expansion and the associated adjoint equation. We don't need anymore the
second order expansion. Moreover, we show that these necessary optimality
conditions for strict-singular controls are also sufficient, without imposing
neither the convexity of $U_{1}$ nor that of the Hamiltonian $H$ in $v$.

The results of this paper are an important improvement of those of Bahlali and
Mezerdi $\left[  2\right]  $ and an extension of the works by Bahlali $\left[
5\right]  $ to the class of singular controls.

The paper is organized as follows. In Section 2, we formulate the
strict-singular and relaxed-singular control problems and give the various
assumptions used throughout the paper. Section 3 is devoted to study the
relaxed-singular control problems and we establish necessary as well as
sufficient conditions of optimality for relaxed-singular controls. In the last
section, we derive directly from the results of Section 3, the optimality
conditions for strict-singular controls.

\ 

Along with this paper, we denote by $C$ some positive constant and for
simplicity, we need the following matrix notations. We denote by
$\mathcal{M}_{n\times d}\left(  \mathbb{R}\right)  $ the space of $n\times d$
real matrix and $\mathcal{M}_{n\times n}^{d}\left(  \mathbb{R}\right)  $ the
linear space of vectors $M=\left(  M_{1},...,M_{d}\right)  $ where $M_{i}%
\in\mathcal{M}_{n\times n}\left(  \mathbb{R}\right)  $. For any $M,N\in
\mathcal{M}_{n\times n}^{d}\left(  \mathbb{R}\right)  $, $L,S\in
\mathcal{M}_{n\times d}\left(  \mathbb{R}\right)  $, $Q\in\mathcal{M}_{n\times
n}\left(  \mathbb{R}\right)  $, $\alpha,\beta\in\mathbb{R}^{n}$ and $\gamma
\in\mathbb{R}^{d},$ we use the following notations

$\alpha\beta=%
{\displaystyle\sum\limits_{i=1}^{n}}
\alpha_{i}\beta_{i}\in\mathbb{R}$ is the product scalar in $\mathbb{R}^{n}$;

$LS=%
{\displaystyle\sum\limits_{i=1}^{d}}
L_{i}S_{i}\in\mathbb{R}$, where $L_{i}$ and\ $S_{i}$ are the $i^{th}$ columns
of $L$ and $S;$

$ML=%
{\displaystyle\sum\limits_{i=1}^{d}}
M_{i}L_{i}\in\mathbb{R}^{n}$;

$M\alpha\gamma=\sum\limits_{i=1}^{d}\left(  M_{i}\alpha\right)  \gamma_{i}%
\in\mathbb{R}^{n}$;

$MN=%
{\displaystyle\sum\limits_{i=1}^{d}}
M_{i}N_{i}\in\mathcal{M}_{n\times n}\left(  \mathbb{R}\right)  $;

$MQN=%
{\displaystyle\sum\limits_{i=1}^{d}}
M_{i}QN_{i}\in\mathcal{M}_{n\times n}\left(  \mathbb{R}\right)  $;

$MQ\gamma=%
{\displaystyle\sum\limits_{i=1}^{d}}
M_{i}Q\gamma_{i}\in\mathcal{M}_{n\times n}\left(  \mathbb{R}\right)  $.

We denote by $L^{\ast}$ the transpose of the matrix $L$ and $M^{\ast}=\left(
M_{1}^{\ast},...,M_{d}^{\ast}\right)  $.

\section{Formulation of the problem}

\ 

Let $\left(  \Omega,\mathcal{F},\left(  \mathcal{F}_{t}\right)  _{t\geq
0},\mathcal{P}\right)  $ be a filtered probability space satisfying the usual
conditions, on which a d-dimensional Brownian motion $W=\left(  W_{t}\right)
_{t\geq0}$\ is defined. We assume that $\left(  \mathcal{F}_{t}\right)  $ is
the $\mathcal{P}-$ augmentation of the natural filtration of $\left(
W_{t}\right)  _{t\geq0}.$

Let $T$ be a strictly positive real number and consider the following sets

$U_{1}$ is a non empty subset of $\mathbb{R}^{k}$,

$U_{2}=\left(  \left[  0,\infty\right)  \right)  ^{m},$

$\mathcal{U}_{1}$ is the class of measurable, adapted processes $v:\left[
0,T\right]  \times\Omega\longrightarrow U_{1}$ such that%
\[
E\left[  \underset{t\in\left[  0,T\right]  }{\sup}\left\vert v_{t}\right\vert
^{2}\right]  <\infty.
\]

$\mathcal{U}_{2}$ is the class of measurable, adapted processes $\eta:\left[
0,T\right]  \times\Omega\longrightarrow U_{2}$ such that $\eta$ is
nondecreasing (componentwise), left-continuous with right limits, $\eta_{0}%
=0$, and
\[
E\left[  \left\vert \eta_{T}\right\vert ^{2}\right]  <\infty.
\]

\subsection{The strict-singular control problem}

\ 

\begin{definition}
\textit{A strict-singular control is a pair of processes }$\left(
v,\eta\right)  \in\mathcal{U}_{1}\times\mathcal{U}_{2}$\textit{.}

\ 

\textit{We denote by }$\mathcal{U=U}_{1}\times\mathcal{U}_{2}$\textit{\ the
set of all strict-singular controls.}
\end{definition}

\ 

For any $\left(  v,\eta\right)  \in\mathcal{U}$, we consider the following
SDE
\begin{equation}
\left\{
\begin{array}
[c]{l}%
dx_{t}^{\left(  v,\eta\right)  }=b\left(  t,x_{t}^{\left(  v,\eta\right)
},v_{t}\right)  dt+\sigma\left(  t,x_{t}^{\left(  v,\eta\right)  }%
,v_{t}\right)  dW_{t}+G_{t}d\eta_{t},\\
x_{0}^{\left(  v,\eta\right)  }=x,
\end{array}
\right.
\end{equation}
where%
\[%
\begin{array}
[c]{l}%
b:\left[  0,T\right]  \times\mathbb{R}^{n}\times U_{1}\longrightarrow
\mathbb{R}^{n},\\
\sigma:\left[  0,T\right]  \mathbb{\times R}^{n}\times U_{1}\longrightarrow
\mathcal{M}_{n\times d}\left(  \mathbb{R}\right)  ,\\
G:\left[  0,T\right]  \longrightarrow\mathcal{M}_{n\times m}\left(
\mathbb{R}\right)  .
\end{array}
\]

The criteria to be minimized is defined from $\mathcal{U}$ into $\mathbb{R}$ by%

\begin{equation}
J\left(  v,\eta\right)  =E\left[  g\left(  x_{T}^{\left(  v,\eta\right)
}\right)  +\int_{0}^{T}h\left(  t,x_{t}^{\left(  v,\eta\right)  }%
,v_{t}\right)  dt+\int_{0}^{T}k_{t}d\eta_{t}\right]  ,
\end{equation}
Where
\begin{align*}
g  &  :\mathbb{R}^{d}\longrightarrow\mathbb{R},\\
h  &  :\left[  0,T\right]  \times\mathbb{R}^{d}\times U_{1}\longrightarrow
\mathbb{R},\\
k  &  :\left[  0,T\right]  \longrightarrow\left(  \lbrack0,\infty)\right)
^{d}.
\end{align*}

A strict-singular control $\left(  v,\eta\right)  $ is called optimal if it
satisfies%
\begin{equation}
J\left(  u,\xi\right)  =\inf\limits_{\left(  v,\eta\right)  \in\mathcal{U}%
}J\left(  v,\eta\right)  .
\end{equation}

\ 

We assume that%
\begin{align}
&  b,\sigma,g\text{ and }h\ \text{are continuously differentiable with respect
to }x\text{.}\nonumber\\
&  \text{The derivatives }b_{x},\sigma_{x},g_{x}\text{ and }h_{x},\text{are
continuous in }\left(  x,v\right)  \text{ and}\nonumber\\
&  \text{uniformly bounded.}\\
&  b\text{ and }\sigma\text{ are bounded by }C\left(  1+\left\vert
x\right\vert +\left\vert v\right\vert \right)  .\nonumber\\
&  G\text{ and }k\text{ are continuous and }G\text{ is bounded}.\nonumber
\end{align}

Under the above assumptions, for every $\left(  v,\eta\right)  \in\mathcal{U}
$, equation $\left(  1\right)  $ has an unique strong solution and the
functional cost $J$ is well defined from $\mathcal{U}$ into $\mathbb{R}$.

\subsection{The relaxed-singular model}

\ 

The strict-singular control problem $\left\{  \left(  1\right)  ,\left(
2\right)  ,\left(  3\right)  \right\}  $ formulated in the last subsection may
fail to have an optimal solution. Let us begin by a deterministic examples
which shows that even in simple cases, existence of a strict optimal control
is not ensured (see Fleming $\left[  16\right]  $ and Yong and Zhou $\left[
28\right]  $ for other examples).

\textbf{Example 1. }The problem is to minimize, over the set $U$ of measurable
functions $v:\left[  0,T\right]  \rightarrow\left\{  -1,1\right\}  $, the
following functional cost%
\[
J\left(  v\right)  =%
{\displaystyle\int\nolimits_{0}^{T}}
\left(  x_{t}^{v}\right)  ^{2}dt,
\]
where $x_{t}^{v}$ denotes the solution of%
\[
\left\{
\begin{array}
[c]{c}%
dx_{t}^{v}=v_{t}dt,\\
x_{0}^{v}=0.
\end{array}
\right.
\]

We have
\[
\inf\limits_{v\in\mathcal{U}}J\left(  v\right)  =0.
\]

Indeed, consider the following sequence of controls%
\[
v_{t}^{n}=\left(  -1\right)  ^{k}\text{ \ if \ }%
\genfrac{.}{.}{}{0}{k}{n}%
T\leq t\leq%
\genfrac{.}{.}{}{0}{k+1}{n}%
T\ \ ,\ \ 0\leq k\leq n-1.
\]

Then clearly
\begin{align*}
\left\vert x_{t}^{v^{n}}\right\vert  &  \leq%
\genfrac{.}{.}{}{0}{T}{n}%
,\\
\left\vert J\left(  v^{n}\right)  \right\vert  &  \leq%
\genfrac{.}{.}{}{0}{T^{3}}{n^{2}}%
.
\end{align*}

Which implies that
\[
\inf\limits_{v\in\mathcal{U}}J\left(  v\right)  =0.
\]

There is however no control $v$ such that $J\left(  v\right)  =0$. If this
would have been the case, then for every $t,\ x_{t}^{v}=0$. This in turn would
imply that $v_{t}=0$, which is impossible. The problem is that the sequence
$\left(  v^{n}\right)  $ has no limit in the space of strict controls. This
limit if it exists, will be the natural candidate for optimality. If we
identify $v_{t}^{n}$ with the Dirac measure $\delta_{v_{t}^{n}}\left(
da\right)  $ and set $q_{n}\left(  dt,dv\right)  =\delta_{v_{t}^{n}}\left(
dv\right)  dt$, we get a measure on $\left[  0,1\right]  \times U$. Then, the
sequence $\left(  q_{n}\left(  dt,dv\right)  \right)  _{n}$ converges weakly
to $%
\genfrac{.}{.}{}{0}{1}{2}%
dt.\left[  \delta_{-1}+\delta_{1}\right]  \left(  da\right)  $.

\textbf{Example 2. }Consider the control problem where the system is governed
by the SDE%
\[
\left\{
\begin{array}
[c]{l}%
dx_{t}=v_{t}dt+dW_{t},\\
x_{0}=0.
\end{array}
\right.
\]

The functional cost to be minimized is given by%
\[
J\left(  v\right)  =\mathbb{E}\int_{0}^{T}\left[  x_{t}^{2}+\left(
1-v_{t}^{2}\right)  ^{2}\right]  dt.
\]

$U=\left[  -1,1\right]  $ and $x,\ v,\ W$ are one dimensional. The control $v$
(open loop) is a measurable function from $\left[  0,T\right]  $ into $U$.

The separation principle applies to this example, the optimal control
minimizes%
\[
\int_{0}^{T}\left[  \widehat{x}_{t}^{2}+\left(  1-v_{t}^{2}\right)
^{2}\right]  ,
\]
where $\widehat{x}_{t}=\mathbb{E}\left[  x_{t}\right]  $ satisfies%
\[
\left\{
\begin{array}
[c]{l}%
d\widehat{x}_{t}=v_{t}dt,\\
\widehat{x}_{0}=0.
\end{array}
\right.
\]

This problem has no optimal strict control. A relaxed solution is to let
\[
\mu_{t}=%
\genfrac{.}{.}{}{}{1}{2}%
\delta_{1}+%
\genfrac{.}{.}{}{}{1}{2}%
\delta_{-1},
\]
where $\delta_{a}$ is an Dirac measure concentrated at a single point $a.$

\ 

This suggests that the set\ of strict controls is too narrow and should be
embedded into a wider class with a richer topological structure for which the
control problem becomes solvable. The idea of relaxed-singular control is to
replace the absolutely continuous part $v_{t}$\ of the strict-singular control
by a $\mathbb{P}\left(  U_{1}\right)  $-valued process $\left(  q_{t}\right)
$, where $\mathbb{P}\left(  U_{1}\right)  $\ is the space of probability
measures on $U_{1}$ equipped with the topology of stable convergence of measures.

\begin{definition}
A relaxed-singular control is a pair $\left(  q,\eta\right)  $ of processes
such that

i) $q$ is a $\mathbb{P}\left(  U_{1}\right)  $-valued process progressively
measurable with respect to $\left(  \mathcal{F}_{t}\right)  $ and such that
for each $t$, $1_{]0,t]}.q$ is $F_{t}$-measurable.

ii) $\eta\in\mathcal{U}_{2}$.

\ 

We denote by $\mathcal{R}=\mathcal{R}_{1}\times\mathcal{U}_{2}$ the set of
relaxed-singular controls.
\end{definition}

\ 

For more details on relaxed controls, see $\left[  3\right]  ,\left[
4\right]  ,\left[  5\right]  ,\left[  15\right]  ,\left[  16\right]  ,\left[
24\right]  ,\left[  25\right]  $ and $\left[  26\right]  $.

\ 

For any $\left(  q,\eta\right)  \in\mathcal{R}$, we consider the following
relaxed-singular SDE%
\begin{equation}
\left\{
\begin{array}
[c]{l}%
dx_{t}^{\left(  q,\eta\right)  }=%
{\displaystyle\int\nolimits_{U_{1}}}
b\left(  t,x_{t}^{\left(  q,\eta\right)  },a\right)  q_{t}\left(  da\right)
dt+%
{\displaystyle\int\nolimits_{U_{1}}}
\sigma\left(  t,x_{t}^{\left(  q,\eta\right)  },a\right)  q_{t}\left(
da\right)  dW_{t}+G_{t}d\eta_{t}\\
x_{0}^{\left(  q,\eta\right)  }=x.
\end{array}
\right.
\end{equation}

The expected cost to be minimized, in the relaxed-singular model, is defined
from $\mathcal{R}$ into $\mathbb{R}$ by%
\begin{equation}
\mathcal{J}\left(  q,\eta\right)  =\mathbb{E}\left[  g\left(  x_{T}^{\left(
q,\eta\right)  }\right)  +\int_{0}^{T}%
{\displaystyle\int\nolimits_{U_{1}}}
h\left(  t,x_{t}^{\left(  q,\eta\right)  },a\right)  q_{t}\left(  da\right)
dt+\int_{0}^{T}k_{t}d\eta_{t}\right]  .
\end{equation}

A relaxed-singular control $\left(  \mu,\xi\right)  $ is called optimal if it
solves%
\begin{equation}
\mathcal{J}\left(  \mu,\xi\right)  =\inf\limits_{\left(  q,\eta\right)
\in\mathcal{R}}\mathcal{J}\left(  q,\eta\right)  .
\end{equation}

Haussmann and Suo $\left[  17\right]  $ have proved that the relaxed-singular
control problem admits an optimal solution under general conditions on the
coefficients. Indeed, by using a compactification method and under some mild
continuity hypotheses on the data, it is shown by purely probabilistic
arguments that an optimal solution for the problem exists. Moreover, the value
function is shown to be Borel measurable. See Haussmann and Suo $\left[
17\right]  $, Section 3, page 925 to page 934 and essentially Theorem 3.8,
page 933.

\begin{remark}
If we put for any $\left(  q,\eta\right)  \in\mathcal{R}$%
\begin{align*}
\overline{b}\left(  t,x_{t}^{\left(  q,\eta\right)  },q_{t}\right)   &  =%
{\displaystyle\int\nolimits_{U_{1}}}
b\left(  t,x_{t}^{\left(  q,\eta\right)  },a\right)  q_{t}\left(  da\right)
,\\
\overline{\sigma}\left(  t,x_{t}^{\left(  q,\eta\right)  },q_{t}\right)   &  =%
{\displaystyle\int\nolimits_{U_{1}}}
\sigma\left(  t,x_{t}^{\left(  q,\eta\right)  },a\right)  q_{t}\left(
da\right)  ,\\
\overline{h}\left(  t,x_{t}^{\left(  q,\eta\right)  },q_{t}\right)   &  =%
{\displaystyle\int\nolimits_{U_{1}}}
h\left(  t,x_{t}^{\left(  q,\eta\right)  },a\right)  q_{t}\left(  da\right)  .
\end{align*}

Then, equation $\left(  5\right)  $ becomes
\[
\left\{
\begin{array}
[c]{l}%
dx_{t}^{\left(  q,\eta\right)  }=\overline{b}\left(  t,x_{t}^{\left(
q,\eta\right)  },q_{t}\right)  dt+\overline{\sigma}\left(  t,x_{t}^{\left(
q,\eta\right)  },q_{t}\right)  dW_{t}+G_{t}d\eta_{t},\\
x_{T}^{\left(  q,\eta\right)  }=x.
\end{array}
\right.
\]

With a functional cost given by
\[
\mathcal{J}\left(  q,\eta\right)  =\mathbb{E}\left[  g\left(  x_{T}^{\left(
q,\eta\right)  }\right)  +\int_{0}^{T}\overline{h}\left(  t,x_{t}^{\left(
q,\eta\right)  },q_{t}\right)  dt+\int_{0}^{T}k_{t}d\eta_{t}\right]  .
\]

Hence, by introducing relaxed-singular controls, we have replaced $U_{1}$ by a
larger space $\mathbb{P}\left(  U_{1}\right)  $. We have gained the advantage
that $\mathbb{P}\left(  U_{1}\right)  $ is convex. Furthermore, the new
coefficients of equation $\left(  5\right)  $ and the running cost are linear
with respect to the relaxed control variable.
\end{remark}

\begin{remark}
The coefficients $\overline{b}\ $and $\overline{\sigma}$ (defined in the above
remark) check respectively the same assumptions as $b$ and $\sigma$. Then,
under assumptions $\left(  4\right)  $, for every $\left(  q,\eta\right)
\in\mathcal{R}$, equation $\left(  5\right)  $ has an unique strong solution.

On the other hand, It is easy to see that $\overline{h}$ checks the same
assumptions as $h$. Then, the functional cost $\mathcal{J}$ is well defined
from $\mathcal{R}$ into $\mathbb{R}$.
\end{remark}

\begin{remark}
If $q_{t}=\delta_{v_{t}}$ is an atomic measure concentrated at a single point
$v_{t}\in\mathbb{P}\left(  U_{1}\right)  $, then for each $t\in\left[
0,T\right]  $ we have%
\begin{align*}
x^{\left(  q,\eta\right)  }  &  =x^{\left(  v,\eta\right)  },\\
\mathcal{J}\left(  q,\eta\right)   &  =J\left(  v,\eta\right)  ,
\end{align*}
and we get a strict-singular control problem. So the problem of
strict-singular controls $\left\{  \left(  1\right)  ,\left(  2\right)
,\left(  3\right)  \right\}  $ is a particular case of relaxed-singular
control problem $\left\{  \left(  5\right)  ,\left(  6\right)  ,\left(
7\right)  \right\}  $.
\end{remark}

\section{Optimality conditions for relaxed-singular controls}

\ 

In this section, we study the problem $\left\{  \left(  5\right)  ,\left(
6\right)  ,\left(  7\right)  \right\}  $ and we establish necessary as well as
sufficient conditions of optimality for relaxed-singular controls.

\subsection{Preliminary results}

\ 

Since the set of relaxed-singular controls $\mathcal{R}$ is convex, a
classical way of treating such a problem is to use the convex perturbation
method. More precisely, let $\left(  \mu,\xi\right)  $ be an optimal
relaxed-singular control and $x_{t}^{\left(  \mu,\xi\right)  }$ the solution
of $\left(  5\right)  $ controlled by $\left(  \mu,\xi\right)  $. Then, for
each $t\in\left[  0,T\right]  $, we can define a perturbed relaxed-singular
control as follows%
\[
\left(  \mu_{t}^{\theta},\xi_{t}^{\theta}\right)  =\left(  \mu_{t},\xi
_{t}\right)  +\theta\left[  \left(  q_{t},\eta_{t}\right)  -\left(  \mu
_{t},\xi_{t}\right)  \right]  ,
\]
where, $\theta>0$ is sufficiently small and $\left(  q,\eta\right)  $ is an
arbitrary element of $\mathcal{R}$.

Denote by $x_{t}^{\left(  \mu^{\theta},\xi^{\theta}\right)  }$ the solution of
$\left(  5\right)  $ associated with $\left(  \mu^{\theta},\xi^{\theta
}\right)  $.

From optimality of $\left(  \mu,\xi\right)  $, the variational inequality will
be derived from the fact that
\[
0\leq\mathcal{J}\left(  \mu^{\theta},\xi^{\theta}\right)  -\mathcal{J}\left(
\mu,\xi\right)  .
\]

For this end, we need the following classical lemmas.

\begin{lemma}
\textit{Under assumptions }$\left(  4\right)  $, we have%
\begin{equation}
\underset{\theta\rightarrow0}{\lim}\left[  \underset{t\in\left[  0,T\right]
}{\sup}\mathbb{E}\left\vert x_{t}^{\left(  \mu^{\theta},\xi^{\theta}\right)
}-x_{t}^{\left(  \mu,\xi\right)  }\right\vert ^{2}\right]  =0.
\end{equation}

\end{lemma}

\begin{proof}
We have
\begin{align*}
&  x_{t}^{\left(  \mu^{\theta},\xi^{\theta}\right)  }-x_{t}^{\left(  \mu
,\xi\right)  }\\
&  =%
{\displaystyle\int\nolimits_{0}^{t}}
\left[
{\displaystyle\int\nolimits_{U_{1}}}
b\left(  s,x_{s}^{\left(  \mu^{\theta},\xi^{\theta}\right)  },a\right)
\mu_{s}^{\theta}\left(  da\right)  -%
{\displaystyle\int\nolimits_{U_{1}}}
b\left(  s,x_{s}^{\left(  \mu,\xi\right)  },a\right)  \mu_{s}^{\theta}\left(
da\right)  \right]  ds\\
&  +%
{\displaystyle\int\nolimits_{0}^{t}}
\left[
{\displaystyle\int\nolimits_{U_{1}}}
b\left(  s,x_{s}^{\left(  \mu,\xi\right)  },a\right)  \mu_{s}^{\theta}\left(
da\right)  -%
{\displaystyle\int\nolimits_{U_{1}}}
b\left(  s,x_{s}^{\left(  \mu,\xi\right)  },a\right)  \mu_{s}\left(
da\right)  \right]  ds\\
&  +%
{\displaystyle\int\nolimits_{0}^{t}}
\left[
{\displaystyle\int\nolimits_{U_{1}}}
\sigma\left(  s,x_{s}^{\left(  \mu^{\theta},\xi^{\theta}\right)  },a\right)
\mu_{s}^{\theta}\left(  da\right)  -%
{\displaystyle\int\nolimits_{U_{1}}}
\sigma\left(  s,x_{s}^{\left(  \mu,\xi\right)  },a\right)  \mu_{s}^{\theta
}\left(  da\right)  \right]  dW_{s}\\
&  +%
{\displaystyle\int\nolimits_{0}^{t}}
\left[
{\displaystyle\int\nolimits_{U_{1}}}
\sigma\left(  s,x_{s}^{\left(  \mu,\xi\right)  },a\right)  \mu_{s}^{\theta
}\left(  da\right)  -%
{\displaystyle\int\nolimits_{U_{1}}}
\sigma\left(  s,x_{s}^{\left(  \mu,\xi\right)  },a\right)  \mu_{s}\left(
da\right)  \right]  dW_{s}\\
&  +%
{\displaystyle\int\nolimits_{0}^{t}}
G_{t}d\left(  \xi_{t}^{\theta}-\xi_{t}\right)  .
\end{align*}

By using the definition of $\left(  \mu^{\theta},\xi^{\theta}\right)  $ and
taking expectation, we have
\begin{align*}
&  \mathbb{E}\left\vert x_{t}^{\left(  \mu^{\theta},\xi^{\theta}\right)
}-x_{t}^{\left(  \mu,\xi\right)  }\right\vert ^{2}\\
&  \leq C\mathbb{E}%
{\displaystyle\int\nolimits_{0}^{t}}
\left\vert
{\displaystyle\int\nolimits_{U_{1}}}
b\left(  s,x_{s}^{\left(  \mu^{\theta},\xi^{\theta}\right)  },a\right)
\mu_{s}\left(  da\right)  -%
{\displaystyle\int\nolimits_{U_{1}}}
b\left(  s,x_{s}^{\left(  \mu,\xi\right)  },a\right)  \mu_{s}\left(
da\right)  \right\vert ^{2}ds\\
&  +C\theta^{2}\mathbb{E}%
{\displaystyle\int\nolimits_{0}^{t}}
\left\vert
{\displaystyle\int\nolimits_{U_{1}}}
b\left(  s,x_{s}^{\left(  \mu^{\theta},\xi^{\theta}\right)  },a\right)
q_{s}\left(  da\right)  -%
{\displaystyle\int\nolimits_{U_{1}}}
b\left(  s,x_{s}^{\left(  \mu^{\theta},\xi^{\theta}\right)  },a\right)
\mu_{s}\left(  da\right)  \right\vert ^{2}ds\\
&  +C\mathbb{E}%
{\displaystyle\int\nolimits_{0}^{t}}
\left\vert
{\displaystyle\int\nolimits_{U_{1}}}
\sigma\left(  s,x_{s}^{\left(  \mu^{\theta},\xi^{\theta}\right)  },a\right)
\mu_{s}\left(  da\right)  -%
{\displaystyle\int\nolimits_{U_{1}}}
\sigma\left(  s,x_{s}^{\left(  \mu,\xi\right)  },a\right)  \mu_{s}\left(
da\right)  \right\vert ^{2}ds\\
&  +C\theta^{2}\mathbb{E}%
{\displaystyle\int\nolimits_{0}^{t}}
\left\vert
{\displaystyle\int\nolimits_{U_{1}}}
\sigma\left(  s,x_{s}^{\left(  \mu^{\theta},\xi^{\theta}\right)  },a\right)
q_{s}\left(  da\right)  -%
{\displaystyle\int\nolimits_{U_{1}}}
\sigma\left(  s,x_{s}^{\left(  \mu^{\theta},\xi^{\theta}\right)  },a\right)
\mu_{s}\left(  da\right)  \right\vert ^{2}ds\\
&  +C\theta\mathbb{E}\left\vert
{\displaystyle\int\nolimits_{0}^{t}}
G_{t}d\left(  \eta_{t}-\xi_{t}\right)  \right\vert ^{2}.
\end{align*}

Since $b$ and $\sigma$ are uniformly Lipschitz with respect to $x$, and $G$ is
bounded, then%
\[
\mathbb{E}\left\vert x_{t}^{\left(  \mu^{\theta},\xi^{\theta}\right)  }%
-x_{t}^{\left(  \mu,\xi\right)  }\right\vert ^{2}\leq C\mathbb{E}%
{\displaystyle\int\nolimits_{0}^{t}}
\left\vert x_{s}^{\left(  \mu^{\theta},\xi^{\theta}\right)  }-x_{s}^{\left(
\mu,\xi\right)  }\right\vert ^{2}ds+C\theta E\left\vert \eta_{T}-\xi
_{T}\right\vert ^{2}+C\theta^{2}.
\]

By using Gronwall's lemma and Buckholder-Davis-Gundy inequality, we obtain the
desired result.
\end{proof}

\begin{lemma}
\textit{Let }$z$\textit{\ be the solution of the linear SDE (called
variational equation)}%
\begin{equation}
\left\{
\begin{array}
[c]{ll}%
dz_{t}= &
{\displaystyle\int\nolimits_{U_{1}}}
b_{x}\left(  t,x_{t}^{\left(  \mu,\xi\right)  },a\right)  \mu_{t}\left(
da\right)  z_{t}dt\\
& +\left[
{\displaystyle\int\nolimits_{U_{1}}}
b\left(  t,x_{t}^{\left(  \mu,\xi\right)  },a\right)  \mu_{t}\left(
da\right)  -%
{\displaystyle\int\nolimits_{U_{1}}}
b\left(  t,x_{t}^{\left(  \mu,\xi\right)  },a\right)  q_{t}\left(  da\right)
\right]  dt\\
& +%
{\displaystyle\int\nolimits_{U_{1}}}
\sigma_{x}\left(  t,x_{t}^{\left(  \mu,\xi\right)  },a\right)  \mu_{t}\left(
da\right)  z_{t}dW_{t}\\
& +\left[
{\displaystyle\int\nolimits_{U_{1}}}
\sigma\left(  t,x_{t}^{\left(  \mu,\xi\right)  },a\right)  \mu_{t}\left(
da\right)  -%
{\displaystyle\int\nolimits_{U_{1}}}
\sigma\left(  t,x_{t}^{\left(  \mu,\xi\right)  },a\right)  q_{t}\left(
da\right)  \right]  dW_{t}\\
& +G_{t}d\left(  \eta_{t}-\xi_{t}\right)  ,\\
z_{0}= & 0.
\end{array}
\right.
\end{equation}

\textit{Then, we have }
\begin{equation}
\underset{\theta\rightarrow0}{\lim}\mathbb{E}\left\vert
\genfrac{.}{.}{}{0}{x_{t}^{\left(  \mu^{\theta},\xi^{\theta}\right)  }%
-x_{t}^{\left(  \mu,\xi\right)  }}{\theta}%
-z_{t}\right\vert ^{2}=0.
\end{equation}

\end{lemma}

\begin{proof}
It is easy to see that
\[%
\genfrac{.}{.}{}{0}{x_{t}^{\left(  \mu^{\theta},\xi^{\theta}\right)  }%
-x_{t}^{\left(  \mu,\xi\right)  }}{\theta}%
-z_{t},
\]
does not depends on the singular part. Then the result follows immediately by
the same method that in $\left[  5,\ \text{Lemma\ 10,\ page 2086-2088}\right]
$.
\end{proof}

\begin{lemma}
\textit{Let }$\left(  \mu,\xi\right)  $\textit{\ be an optimal
relaxed-singular control minimizing the cost }$\mathcal{J}$\textit{\ over
}$\mathcal{R}$\textit{\ and }$x_{t}^{\left(  \mu,\xi\right)  }$\textit{\ the
associated optimal trajectory. Then, for any }$\left(  q,\eta\right)
\in\mathcal{R}$,\textit{\ we have}
\begin{align}
0  &  \leq\mathbb{E}\left[  g_{x}\left(  x_{T}^{\left(  \mu,\xi\right)
}\right)  z_{T}\right]  +\mathbb{E}%
{\displaystyle\int\nolimits_{0}^{T}}
{\displaystyle\int\nolimits_{U_{1}}}
h_{x}\left(  t,x_{t}^{\left(  \mu,\xi\right)  },a\right)  \mu_{t}\left(
da\right)  z_{t}dt\\
&  +\mathbb{E}%
{\displaystyle\int\nolimits_{0}^{T}}
\left[
{\displaystyle\int\nolimits_{U_{1}}}
h\left(  t,x_{t}^{\left(  \mu,\xi\right)  },a\right)  q_{t}\left(  da\right)
-%
{\displaystyle\int\nolimits_{U_{1}}}
h\left(  t,x_{t}^{\left(  \mu,\xi\right)  },a\right)  \mu_{t}\left(
da\right)  \right]  dt\nonumber\\
&  +E\int_{0}^{T}k_{t}d\left(  \eta_{t}-\xi_{t}\right)  .\nonumber
\end{align}

\end{lemma}

\begin{proof}
Since $\left(  \mu,\xi\right)  $ minimizes the cost $\mathcal{J}%
$\textit{\ over }$\mathcal{R}$, then%
\begin{align*}
0  &  \leq\mathcal{J}\left(  \mu^{\theta},\xi^{\theta}\right)  -\mathcal{J}%
\left(  \mu,\xi\right) \\
&  \leq\mathbb{E}\left[  g\left(  x_{T}^{\left(  \mu^{\theta},\xi^{\theta
}\right)  }\right)  -g\left(  x_{T}^{\left(  \mu,\xi\right)  }\right)  \right]
\\
&  +\mathbb{E}%
{\displaystyle\int\nolimits_{0}^{T}}
{\displaystyle\int\nolimits_{U_{1}}}
h\left(  t,x_{t}^{\left(  \mu^{\theta},\xi^{\theta}\right)  },a\right)
\mu_{t}^{\theta}\left(  da\right)  dt-\mathbb{E}%
{\displaystyle\int\nolimits_{0}^{T}}
{\displaystyle\int\nolimits_{U_{1}}}
h\left(  t,x_{t}^{\left(  \mu,\xi\right)  },a\right)  \mu_{t}^{\theta}\left(
da\right)  dt\\
&  +\mathbb{E}%
{\displaystyle\int\nolimits_{0}^{T}}
{\displaystyle\int\nolimits_{U_{1}}}
h\left(  t,x_{t}^{\left(  \mu,\xi\right)  },a\right)  \mu_{t}^{\theta}\left(
da\right)  dt-\mathbb{E}%
{\displaystyle\int\nolimits_{0}^{T}}
{\displaystyle\int\nolimits_{U_{1}}}
h\left(  t,x_{t}^{\left(  \mu,\xi\right)  },a\right)  \mu_{t}\left(
da\right)  dt\\
&  +\mathbb{E}%
{\displaystyle\int\nolimits_{0}^{T}}
k_{t}d\left(  \xi_{t}^{\theta}-\xi_{t}\right)  .
\end{align*}

By using the definition of $\left(  \mu^{\theta},\xi^{\theta}\right)  $, we
get%
\begin{align*}
0  &  \leq\mathbb{E}\left[  g\left(  x_{T}^{\left(  \mu^{\theta},\xi^{\theta
}\right)  }\right)  -g\left(  x_{T}^{\left(  \mu,\xi\right)  }\right)  \right]
\\
&  +\mathbb{E}%
{\displaystyle\int\nolimits_{0}^{T}}
\left[
{\displaystyle\int\nolimits_{U_{1}}}
h\left(  t,x_{t}^{\left(  \mu^{\theta},\xi^{\theta}\right)  },a\right)
\mu_{t}\left(  da\right)  -%
{\displaystyle\int\nolimits_{U_{1}}}
h\left(  t,x_{t}^{\left(  \mu,\xi\right)  },a\right)  \mu_{t}\left(
da\right)  \right]  dt\\
&  +\theta\mathbb{E}%
{\displaystyle\int\nolimits_{0}^{T}}
\left[
{\displaystyle\int\nolimits_{U_{1}}}
h\left(  t,x_{t}^{\left(  \mu^{\theta},\xi^{\theta}\right)  },a\right)
q_{t}\left(  da\right)  -%
{\displaystyle\int\nolimits_{U_{1}}}
h\left(  t,x_{t}^{\left(  \mu^{\theta},\xi^{\theta}\right)  },a\right)
\mu_{t}\left(  da\right)  \right]  dt\\
&  +\theta\mathbb{E}%
{\displaystyle\int\nolimits_{0}^{T}}
k_{t}d\left(  \eta_{t}-\xi_{t}\right)  .
\end{align*}

Hence,%
\begin{align}
0  &  \leq\mathbb{E}%
{\displaystyle\int\nolimits_{0}^{1}}
g_{x}\left(  x_{T}^{\left(  \mu,\xi\right)  }+\lambda\theta\left(
X_{T}^{\theta}+z_{T}\right)  \right)  z_{T}d\lambda\\
&  +\mathbb{E}%
{\displaystyle\int\nolimits_{0}^{T}}
{\displaystyle\int\nolimits_{U_{1}}}
{\displaystyle\int\nolimits_{0}^{1}}
h_{x}\left(  t,x_{t}^{\left(  \mu,\xi\right)  }+\lambda\theta\left(
X_{t}^{\theta}+z_{t}\right)  ,a\right)  \mu_{t}\left(  da\right)
z_{t}d\lambda dt\nonumber\\
&  +\mathbb{E}%
{\displaystyle\int\nolimits_{0}^{T}}
\left[
{\displaystyle\int\nolimits_{U_{1}}}
h\left(  t,x_{t}^{\left(  \mu,\xi\right)  },a\right)  q_{t}\left(  da\right)
-%
{\displaystyle\int\nolimits_{U_{1}}}
h\left(  t,x_{t}^{\left(  \mu,\xi\right)  },a\right)  \mu_{t}\left(
da\right)  \right]  dt\nonumber\\
&  +\mathbb{E}%
{\displaystyle\int\nolimits_{0}^{T}}
k_{t}d\left(  \eta_{t}-\xi_{t}\right)  +\rho_{t}^{\theta},\nonumber
\end{align}
where%
\[
X_{t}^{\theta}=%
\genfrac{.}{.}{}{0}{x_{t}^{\left(  \mu^{\theta},\xi^{\theta}\right)  }%
-x_{t}^{\left(  \mu,\xi\right)  }}{\theta}%
-z_{t},
\]
and $\rho_{t}^{\theta}$ is given by%
\begin{align*}
&  \rho_{t}^{\theta}=\mathbb{E}%
{\displaystyle\int\nolimits_{0}^{1}}
g_{x}\left(  x_{T}^{\left(  \mu,\xi\right)  }+\lambda\theta\left(
X_{T}^{\theta}+z_{T}\right)  \right)  X_{T}^{\theta}d\lambda\\
&  +\mathbb{E}%
{\displaystyle\int\nolimits_{0}^{T}}
{\displaystyle\int\nolimits_{U_{1}}}
{\displaystyle\int\nolimits_{0}^{1}}
h_{x}\left(  t,x_{t}^{\left(  \mu,\xi\right)  }+\lambda\theta\left(
X_{t}^{\theta}+z_{t}\right)  ,a\right)  \mu_{t}\left(  da\right)
X_{t}^{\theta}d\lambda dt.
\end{align*}

By $\left(  10\right)  $, we have
\[
\underset{\theta\rightarrow0}{\lim}\mathbb{E}\left\vert X_{t}^{\theta
}\right\vert ^{2}=0.
\]

Since $g_{x}$ and $h_{x}$ are continuous and bounded, then by using the
Cauchy-Schwartz inequality we get%
\[
\underset{\theta\rightarrow0}{\lim}\rho_{t}^{\theta}=0,
\]
and by letting $\theta$ go to $0$ in $\left(  12\right)  $, the proof is completed.
\end{proof}

\subsection{Variational inequality and adjoint equation}

\ 

In this subsection, we introduce the adjoint process. With this process, we
derive the variational inequality from $\left(  11\right)  $. The linear terms
in $\left(  11\right)  $ may be treated in the following way. Let $\Phi$ be
the fundamental solution of the linear SDE
\[
\left\{
\begin{array}
[c]{ll}%
d\Phi_{t}= &
{\displaystyle\int\nolimits_{U_{1}}}
b_{x}\left(  t,x_{t}^{\left(  \mu,\xi\right)  },a\right)  \mu_{t}\left(
da\right)  \Phi_{t}dt+%
{\displaystyle\int\nolimits_{U_{1}}}
\sigma_{x}\left(  t,x_{t}^{\left(  \mu,\xi\right)  },a\right)  \mu_{t}\left(
da\right)  \Phi_{t}dW_{t},\\
\Phi_{0}= & I_{d}.
\end{array}
\right.
\]

This equation is linear with bounded coefficients. Hence, it admits an unique
strong solution which is invertible, and its inverse $\Psi_{t}$ is the unique
solution of%
\[
\left\{
\begin{array}
[c]{ll}%
d\Psi_{t}= & \left[
{\displaystyle\int\nolimits_{U_{1}}}
\sigma_{x}\left(  t,x_{t}^{\left(  \mu,\xi\right)  },a\right)  \mu_{t}\left(
da\right)  \Psi_{t}%
{\displaystyle\int\nolimits_{U_{1}}}
\sigma_{x}^{\ast}\left(  t,x_{t}^{\left(  \mu,\xi\right)  },a\right)  \mu
_{t}\left(  da\right)  \right]  dt\\
& -%
{\displaystyle\int\nolimits_{U_{1}}}
b_{x}\left(  t,x_{t}^{\left(  \mu,\xi\right)  },a\right)  \mu_{t}\left(
da\right)  \Psi_{t}dt\\
& -%
{\displaystyle\int\nolimits_{U_{1}}}
\sigma_{x}\left(  t,x_{t}^{\left(  \mu,\xi\right)  },a\right)  \mu_{t}\left(
da\right)  \Psi_{t}dW_{t},\\
\Psi_{0}= & I_{d}.
\end{array}
\right.
\]

Moreover, $\Phi$ and $\Psi$ satisfy
\begin{equation}
\mathbb{E}\left[  \underset{t\in\left[  0,T\right]  }{\sup}\left\vert \Phi
_{t}\right\vert ^{2}\right]  +\mathbb{E}\left[  \underset{t\in\left[
0,T\right]  }{\sup}\left\vert \Psi_{t}\right\vert ^{2}\right]  <\infty.
\end{equation}

We introduce the following processes
\begin{align}
\alpha_{t}  &  =\Psi_{t}z_{t},\\
X  &  =\Phi_{T}^{\ast}g_{x}(x_{T}^{\left(  \mu,\xi\right)  })+\int_{0}^{T}%
{\displaystyle\int\nolimits_{U_{1}}}
\Phi_{t}^{\ast}h_{x}\left(  t,x_{t}^{\left(  \mu,\xi\right)  },a\right)
\mu_{t}\left(  da\right)  dt,\\
Y_{t}  &  =\mathbb{E}\left[  X\;/\;\mathcal{F}_{t}\right]  -\int_{0}^{t}%
{\displaystyle\int\nolimits_{U_{1}}}
\Phi_{s}^{\ast}h_{x}\left(  s,x_{s}^{\left(  \mu,\xi\right)  },a\right)
\mu_{s}\left(  da\right)  ds.
\end{align}

We remark from $\left(  14\right)  ,\left(  15\right)  $ and $\left(
16\right)  $ that
\begin{equation}
\mathbb{E}\left[  \alpha_{T}Y_{T}\right]  =\mathbb{E}\left[  g_{x}\left(
x_{T}^{\left(  \mu,\xi\right)  }\right)  z_{T}\right]  .
\end{equation}

Since $g_{x}$ and $h_{x}$ are bounded, then by $\left(  13\right)  $, $X$ is
square integrable. Hence, the process $\left(  \mathbb{E}\left[
X\ /\ \mathcal{F}_{t}\right]  \right)  _{t\geq0}$ is a square integrable
martingale with respect to the natural filtration of the Brownian motion $W$.
Then, by It\^{o}'s representation theorem we have
\[
Y_{t}=\mathbb{E}\left[  X\right]  +\int_{0}^{t}Q_{s}dW_{s}-\int_{0}^{t}%
{\displaystyle\int\nolimits_{U_{1}}}
\Phi_{s}^{\ast}h_{x}\left(  s,x_{s}^{\left(  \mu,\xi\right)  },a\right)
\mu_{s}\left(  da\right)  ds,
\]
where, $Q$ is an adapted process such that $\mathbb{E}%
{\displaystyle\int_{0}^{T}}
\left\vert Q_{s}\right\vert ^{2}ds<\infty.$

\ 

By applying It\^{o}'s formula to $\alpha_{t}$ then with $\alpha_{t}Y_{t}$ and
using $\left(  17\right)  $, the variational inequality $\left(  11\right)  $
becomes%
\begin{align}
0  &  \leq\mathbb{E}\int_{0}^{T}\left[  \mathcal{H}\left(  t,x_{t}^{\left(
\mu,\xi\right)  },q_{t},p_{t}^{\left(  \mu,\xi\right)  },P_{t}^{\left(
\mu,\xi\right)  }\right)  -\mathcal{H}\left(  t,x_{t}^{\left(  \mu,\xi\right)
},\mu_{t},p_{t}^{\left(  \mu,\xi\right)  },P_{t}^{\left(  \mu,\xi\right)
}\right)  \right]  dt\\
&  +\mathbb{E}\int_{0}^{T}\left(  k_{t}+G_{t}^{\ast}p_{t}\right)  d\left(
\eta_{t}-\xi_{t}\right)  ,\nonumber
\end{align}
where, the Hamiltonian $\mathcal{H}$ is defined from $\left[  0,T\right]
\times\mathbb{R}^{n}\times\mathbb{P}\left(  U_{1}\right)  \times\mathbb{R}%
^{n}\times\mathcal{M}_{n\times d}\left(  \mathbb{R}\right)  $ into $\mathbb{R}
$ by
\[
\mathcal{H}\left(  t,x,q,p,P\right)  =%
{\displaystyle\int\nolimits_{U_{1}}}
h\left(  t,x,a\right)  q\left(  da\right)  +%
{\displaystyle\int\nolimits_{U_{1}}}
b\left(  t,x,a\right)  q\left(  da\right)  p+%
{\displaystyle\int\nolimits_{U_{1}}}
\sigma\left(  t,x,a\right)  q\left(  da\right)  P,
\]
$\left(  p^{\left(  \mu,\xi\right)  },P^{\left(  \mu,\xi\right)  }\right)  $
is a pair of adapted processes given by
\begin{align}
p_{t}^{\left(  \mu,\xi\right)  }  &  =\Psi_{t}^{\ast}Y_{t},\ \ p^{\left(
\mu,\xi\right)  }\in\mathcal{L}^{2}\left(  \left[  0,T\right]  ;\mathbb{R}%
^{n}\right) \\
P_{t}^{\left(  \mu,\xi\right)  }  &  =\Psi_{t}^{\ast}Q_{t}-%
{\displaystyle\int\nolimits_{U_{1}}}
\sigma_{x}^{\ast}\left(  t,x_{t}^{\left(  \mu,\xi\right)  },a\right)  \mu
_{t}\left(  da\right)  p_{t}^{\left(  \mu,\xi\right)  },\ P^{\left(  \mu
,\xi\right)  }\in\mathcal{L}^{2}\left(  \left[  0,T\right]  ;\mathbb{R}%
^{n\times d}\right)  ,
\end{align}
and the process $Q$\ satisfies
\begin{align*}
\int_{0}^{t}Q_{s}dWs  &  =\mathbb{E}\left[  \Phi_{T}^{\ast}g_{x}%
(x_{T}^{\left(  \mu,\xi\right)  })+\int_{0}^{T}\Phi_{t}^{\ast}%
{\displaystyle\int\nolimits_{U_{1}}}
h_{x}\left(  t,x_{t}^{\left(  \mu,\xi\right)  },a\right)  \mu_{t}\left(
da\right)  dt\;/\;\mathcal{F}_{t}\right] \\
&  -\mathbb{E}\left[  \Phi_{T}^{\ast}g_{x}(x_{T}^{\left(  \mu,\xi\right)
})+\int_{0}^{T}\Phi_{t}^{\ast}%
{\displaystyle\int\nolimits_{U_{1}}}
h_{x}\left(  t,x_{t}^{\left(  \mu,\xi\right)  },a\right)  \mu_{t}\left(
da\right)  dt\right]  .
\end{align*}

\ 

The process $p^{\left(  \mu,\xi\right)  }$ is called the adjoint process and
from $\left(  15\right)  ,$ $\left(  16\right)  $ and $\left(  19\right)  $,
it is given explicitly by
\[
p_{t}^{\left(  \mu,\xi\right)  }=\mathbb{E}\left[  \Psi_{t}^{\ast}\Phi
_{T}^{\ast}g_{x}(x_{T}^{\left(  \mu,\xi\right)  })+\Psi_{t}^{\ast}\int_{t}^{T}%
{\displaystyle\int\nolimits_{U_{1}}}
\Phi_{s}^{\ast}h_{x}\left(  s,x_{s}^{\left(  \mu,\xi\right)  },a\right)
\mu_{s}\left(  da\right)  ds\ /\;\mathcal{F}_{t}\right]  .
\]

By applying It\^{o}'s formula to the adjoint processes $p^{\left(  \mu
,\xi\right)  }$ in $\left(  19\right)  $, we obtain the adjoint equation,
which is a linear backward SDE, given by%
\begin{equation}
\left\{
\begin{array}
[c]{l}%
dp_{t}^{\left(  \mu,\xi\right)  }=-\mathcal{H}_{x}\left(  t,x_{t}^{\left(
\mu,\xi\right)  },\mu_{t},p_{t}^{\left(  \mu,\xi\right)  },P_{t}^{\left(
\mu,\xi\right)  }\right)  dt+P_{t}^{\left(  \mu,\xi\right)  }dW_{t},\\
p_{T}^{\left(  \mu,\xi\right)  }=g_{x}(x_{T}^{\left(  \mu,\xi\right)  }).
\end{array}
\right.
\end{equation}

\subsection{Necessary optimality conditions for relaxed-singular controls}

\ 

Starting from the variational inequality $\left(  18\right)  $, we can now
state the necessary optimality conditions, for the relaxed-singular control
problem $\left\{  \left(  5\right)  ,\left(  6\right)  ,\left(  7\right)
\right\}  $, in integral form.

\begin{theorem}
(Necessary optimality conditions for relaxed-singular controls in integral
form). \textit{Let }$\left(  \mu,\xi\right)  $\textit{\ be an optimal
relaxed-singular control minimizing the cost }$\mathcal{J}$\textit{\ over
}$\mathcal{R}$\textit{\ and }$x^{\left(  \mu,\xi\right)  }$\textit{\ denotes
the corresponding optimal trajectory. Then, there exists an unique pair of
adapted processes}
\[
\left(  p^{\left(  \mu,\xi\right)  },P^{\left(  \mu,\xi\right)  }\right)
\in\mathcal{L}^{2}\left(  \left[  0,T\right]  ;\mathbb{R}^{n}\right)
\times\mathcal{L}^{2}\left(  \left[  0,T\right]  ;\mathbb{R}^{n\times
d}\right)  ,
\]
\textit{solution of the backward SDE }$\left(  21\right)  $\textit{\ }such
that, for every $\left(  q,\eta\right)  \in\mathcal{R}$%
\begin{align}
0  &  \leq\mathbb{E}\int_{0}^{T}\left[  \mathcal{H}\left(  t,x_{t}^{\left(
\mu,\xi\right)  },q_{t},p_{t}^{\left(  \mu,\xi\right)  },P_{t}^{\left(
\mu,\xi\right)  }\right)  -\mathcal{H}\left(  t,x_{t}^{\left(  \mu,\xi\right)
},\mu_{t},p_{t}^{\left(  \mu,\xi\right)  },P_{t}^{\left(  \mu,\xi\right)
}\right)  \right]  dt\\
&  +\mathbb{E}\int_{0}^{T}\left(  k_{t}+G_{t}^{\ast}p_{t}^{\left(  \mu
,\xi\right)  }\right)  d\left(  \eta_{t}-\xi_{t}\right)  ,\nonumber
\end{align}

\end{theorem}

\begin{proof}
The result follows immediately from $\left(  18\right)  $.
\end{proof}

We are ready now state necessary optimality conditions for the
relaxed-singular control problem $\left\{  \left(  5\right)  ,\left(
6\right)  ,\left(  7\right)  \right\}  $, in global form.

\begin{theorem}
(Necessary optimality conditions for relaxed-singular controls in global
form). \textit{Let }$\left(  \mu,\xi\right)  $\textit{\ be an optimal
relaxed-singular control minimizing the cost }$\mathcal{J}$\textit{\ over
}$\mathcal{R}$\textit{\ and }$x^{\left(  \mu,\xi\right)  }$\textit{\ denotes
the corresponding optimal trajectory. Then, there exists an unique pair of
adapted processes}
\[
\left(  p^{\left(  \mu,\xi\right)  },P^{\left(  \mu,\xi\right)  }\right)
\in\mathcal{L}^{2}\left(  \left[  0,T\right]  ;\mathbb{R}^{n}\right)
\times\mathcal{L}^{2}\left(  \left[  0,T\right]  ;\mathbb{R}^{n\times
d}\right)  ,
\]
\textit{solution of the backward SDE }$\left(  21\right)  $\textit{\ }such
that%
\begin{equation}
\mathcal{H}\left(  t,x_{t}^{\left(  \mu,\xi\right)  },\mu_{t},p_{t}^{\left(
\mu,\xi\right)  },P_{t}^{\left(  \mu,\xi\right)  }\right)  =\underset{q_{t}%
\in\mathbb{P}\left(  U_{1}\right)  }{\inf}\mathcal{H}\left(  t,x_{t}^{\left(
\mu,\xi\right)  },q_{t},p_{t}^{\left(  \mu,\xi\right)  },P_{t}^{\left(
\mu,\xi\right)  }\right)  ,\ a.e,\ a.s,
\end{equation}%
\begin{equation}
\mathcal{P}\left\{  \forall t\in\left[  0,T\right]  ,\;\forall i\;;\;\left(
k_{_{i}}\left(  t\right)  +G_{i}^{\ast}\left(  t\right)  .p_{t}^{\left(
\mu,\xi\right)  }\right)  \geq0\right\}  =1,
\end{equation}%
\begin{equation}
\mathcal{P}\left\{  \sum_{i=1}^{d}\mathbf{1}_{\left\{  k_{_{i}}\left(
t\right)  +G_{i}^{\ast}\left(  t\right)  p_{t}^{\left(  \mu,\xi\right)  }%
\geq0\right\}  }d\xi_{t}^{i}=0\right\}  =1.
\end{equation}

\end{theorem}

\begin{proof}
Let $\left(  \mu,\xi\right)  $ be an optimal solution of problem $\left\{
\left(  5\right)  ,\left(  6\right)  ,\left(  7\right)  \right\}  $. The
necessary optimality conditions in integral form $\left(  22\right)  $ is
valid for every $\left(  q,\eta\right)  \in\mathcal{R}$. If we put in $\left(
22\right)  $ $\eta=\xi$, then $\left(  23\right)  $ becomes immediately. On
the other hand, if we choose in $\left(  22\right)  $ $q=\mu$, then we can
show $\left(  24\right)  $ and $\left(  25\right)  $, by the same method that
in $\left[  8,\ \text{Theorem 4.2}\right]  $ or $\left[  4,\text{\ Theorem
3.7}\right]  $.
\end{proof}

\subsection{Sufficient optimality conditions for relaxed-singular controls}

\ 

\begin{theorem}
(Sufficient optimality conditions for relaxed-singular controls). Assume that
the functions $g$\ and $x\longmapsto\mathcal{H}\left(  t,x,q,p,P\right)
$\ are convex. Then, $\left(  \mu,\xi\right)  $\ is an optimal solution of
problem $\left\{  \left(  5\right)  ,\left(  6\right)  ,\left(  7\right)
\right\}  $, if it satisfies $\left(  23\right)  ,\ \left(  24\right)  $ and
$\left(  25\right)  .$
\end{theorem}

\begin{proof}
We know that the set of relaxed-singular controls $\mathcal{R}$\ is convex and
the Hamiltonian $\mathcal{H}$\ is linear in $q$.

Let $\left(  \mu,\xi\right)  $\ be an arbitrary element of $\mathcal{R}%
$\ (candidate to be optimal). For any $\left(  q,\eta\right)  \in\mathcal{R}$,
we have%
\begin{align*}
\mathcal{J}\left(  \mu,\xi\right)  -\mathcal{J}\left(  q,\eta\right)   &
=\mathbb{E}\left[  g\left(  x_{T}^{\left(  \mu,\xi\right)  }\right)  -g\left(
x_{T}^{\left(  q,\eta\right)  }\right)  \right] \\
&  +\mathbb{E}%
{\displaystyle\int\nolimits_{0}^{T}}
\left[
{\displaystyle\int\nolimits_{U_{1}}}
h\left(  t,x_{t}^{\left(  \mu,\xi\right)  },a\right)  \mu_{t}\left(
da\right)  -%
{\displaystyle\int\nolimits_{U_{1}}}
h\left(  t,x_{t}^{\left(  q,\eta\right)  },a\right)  q_{t}\left(  da\right)
\right]  dt\\
&  +\mathbb{E}%
{\displaystyle\int\nolimits_{0}^{T}}
k_{t}d\left(  \xi_{t}-\eta_{t}\right)  .
\end{align*}

Since $g$\ is convex, we get%
\[
g\left(  x_{T}^{\left(  q,\eta\right)  }\right)  -g\left(  x_{T}^{\left(
\mu,\xi\right)  }\right)  \geq g_{x}\left(  x_{T}^{\left(  \mu,\xi\right)
}\right)  \left(  x_{T}^{\left(  q,\eta\right)  }-x_{T}^{\left(  \mu
,\xi\right)  }\right)  .
\]

Thus,%
\[
g\left(  x_{T}^{\left(  \mu,\xi\right)  }\right)  -g\left(  x_{T}^{\left(
q,\eta\right)  }\right)  \leq g_{x}\left(  x_{T}^{\left(  \mu,\xi\right)
}\right)  \left(  x_{T}^{\left(  \mu,\xi\right)  }-x_{T}^{\left(
q,\eta\right)  }\right)  .
\]

Hence,%
\begin{align*}
\mathcal{J}\left(  \mu,\xi\right)  -\mathcal{J}\left(  q,\eta\right)   &
\leq\mathbb{E}\left[  g_{x}\left(  x_{T}^{\left(  \mu,\xi\right)  }\right)
\left(  x_{T}^{\left(  \mu,\xi\right)  }-x_{T}^{\left(  q,\eta\right)
}\right)  \right] \\
&  +\mathbb{E}%
{\displaystyle\int\nolimits_{0}^{T}}
\left[
{\displaystyle\int\nolimits_{U_{1}}}
h\left(  t,x_{t}^{\left(  \mu,\xi\right)  },a\right)  \mu_{t}\left(
da\right)  -%
{\displaystyle\int\nolimits_{U_{1}}}
h\left(  t,x_{t}^{\left(  q,\eta\right)  },a\right)  q_{t}\left(  da\right)
\right]  dt\\
&  +\mathbb{E}%
{\displaystyle\int\nolimits_{0}^{T}}
k_{t}d\left(  \xi_{t}-\eta_{t}\right)  .
\end{align*}

We remark that $p_{T}^{\left(  \mu,\xi\right)  }=g_{x}\left(  x_{T}^{\left(
\mu,\xi\right)  }\right)  $, then we have
\begin{align*}
\mathcal{J}\left(  \mu,\xi\right)  -\mathcal{J}\left(  q,\eta\right)   &
\leq\mathbb{E}\left[  p_{T}^{\left(  \mu,\xi\right)  }\left(  x_{T}^{\left(
\mu,\xi\right)  }-x_{T}^{\left(  q,\eta\right)  }\right)  \right] \\
&  +\mathbb{E}%
{\displaystyle\int\nolimits_{0}^{T}}
\left[
{\displaystyle\int\nolimits_{U_{1}}}
h\left(  t,x_{t}^{\left(  \mu,\xi\right)  },a\right)  \mu_{t}\left(
da\right)  -%
{\displaystyle\int\nolimits_{U_{1}}}
h\left(  t,x_{t}^{\left(  q,\eta\right)  },a\right)  q_{t}\left(  da\right)
\right]  dt\\
&  +\mathbb{E}%
{\displaystyle\int\nolimits_{0}^{T}}
k_{t}d\left(  \xi_{t}-\eta_{t}\right)  .
\end{align*}

Applying It\^{o}'s formula to $p_{t}^{\left(  \mu,\xi\right)  }\left(
x_{t}^{\left(  \mu,\xi\right)  }-x_{t}^{\left(  q,\eta\right)  }\right)
$\ and taking expectation, we obtain%
\begin{align}
&  \mathcal{J}\left(  \mu,\xi\right)  -\mathcal{J}\left(  q,\eta\right) \\
&  \leq\mathbb{E}%
{\displaystyle\int\nolimits_{0}^{T}}
\left[  \mathcal{H}\left(  t,x_{t}^{\left(  \mu,\xi\right)  },\mu_{t}%
,p_{t}^{\left(  \mu,\xi\right)  },P_{t}^{\left(  \mu,\xi\right)  }\right)
-\mathcal{H}\left(  t,x_{t}^{\left(  q,\eta\right)  },q_{t},p_{t}^{\left(
\mu,\xi\right)  },P_{t}^{\left(  \mu,\xi\right)  }\right)  \right]
dt\nonumber\\
&  -\mathbb{E}%
{\displaystyle\int\nolimits_{0}^{T}}
\mathcal{H}_{x}\left(  t,x_{t}^{\left(  \mu,\xi\right)  },\mu_{t}%
,p_{t}^{\left(  \mu,\xi\right)  },P_{t}^{\left(  \mu,\xi\right)  }\right)
\left(  x_{t}^{\left(  \mu,\xi\right)  }-x_{t}^{\left(  q,\eta\right)
}\right)  dt\nonumber\\
&  +\mathbb{E}%
{\displaystyle\int\nolimits_{0}^{T}}
\left(  k_{t}+G_{t}^{\ast}p_{t}^{\left(  \mu,\xi\right)  }\right)  d\left(
\xi_{t}-\eta_{t}\right)  .\nonumber
\end{align}

Since $\mathcal{H}$\ is convex in $x$\ and linear in $\mu$, then by using the
Clarke generalized gradient of $\mathcal{H}$\ evaluated at $\left(  x_{t}%
,\mu_{t}\right)  $\ and $\left(  23\right)  $, it follows that
\begin{align*}
&  \mathcal{H}\left(  t,x_{t}^{\left(  q,\eta\right)  },q_{t},p_{t}^{\left(
\mu,\xi\right)  },P_{t}^{\left(  \mu,\xi\right)  }\right)  -\mathcal{H}\left(
t,x_{t}^{\left(  \mu,\xi\right)  },\mu_{t},p_{t}^{\left(  \mu,\xi\right)
},P_{t}^{\left(  \mu,\xi\right)  }\right) \\
&  \geq\mathcal{H}_{x}\left(  t,x_{t}^{\left(  \mu,\xi\right)  },\mu_{t}%
,p_{t}^{\left(  \mu,\xi\right)  },P_{t}^{\left(  \mu,\xi\right)  }\right)
\left(  x_{t}^{\left(  q,\eta\right)  }-x_{t}^{\left(  \mu,\xi\right)
}\right)  .
\end{align*}

Or equivalently%
\begin{align*}
0  &  \geq\mathcal{H}\left(  t,x_{t}^{\left(  \mu,\xi\right)  },\mu_{t}%
,p_{t}^{\left(  \mu,\xi\right)  },P_{t}^{\left(  \mu,\xi\right)  }\right)
-\mathcal{H}\left(  t,x_{t}^{\left(  q,\eta\right)  },q_{t},p_{t}^{\left(
\mu,\xi\right)  },P_{t}^{\left(  \mu,\xi\right)  }\right) \\
&  -\mathcal{H}_{x}\left(  t,x_{t}^{\left(  \mu,\xi\right)  },\mu_{t}%
,p_{t}^{\left(  \mu,\xi\right)  },P_{t}^{\left(  \mu,\xi\right)  }\right)
\left(  x_{t}^{\left(  \mu,\xi\right)  }-x_{t}^{\left(  q,\eta\right)
}\right)  .
\end{align*}

By the above inequality and $\left(  26\right)  $, we have%
\[
\mathcal{J}\left(  \mu,\xi\right)  -\mathcal{J}\left(  q,\eta\right)
\leq\mathbb{E}%
{\displaystyle\int\nolimits_{0}^{T}}
\left(  k_{t}+G_{t}^{\ast}p_{t}^{\left(  \mu,\xi\right)  }\right)  d\left(
\xi_{t}-\eta_{t}\right)  .
\]

By $\left(  24\right)  $ and $\left(  25\right)  $, we show that
\[
\mathbb{E}%
{\displaystyle\int\nolimits_{0}^{T}}
\left(  k_{t}+G_{t}^{\ast}p_{t}^{\left(  \mu,\xi\right)  }\right)  d\left(
\eta_{t}-\xi_{t}\right)  \geq0.
\]

Then, we get%
\[
\mathcal{J}\left(  \mu,\xi\right)  -\mathcal{J}\left(  q,\eta\right)  \leq0.
\]

The theorem is proved.
\end{proof}

\section{Optimality conditions for strict-singular controls}

\ 

In this section, we study the strict-singular control problem $\left\{
\left(  1\right)  ,\left(  2\right)  ,\left(  3\right)  \right\}  $ and from
the results of section 3, we derive the optimality conditions for
strict-singular controls.

\ 

Throughout this section and in addition to the assumptions $\left(  4\right)
$, we suppose that
\begin{align}
&  U_{1}\text{ is compact,}\\
&  b,\ \sigma\text{ and }h\text{ are\ bounded.}%
\end{align}

Consider the following subset of $\mathcal{R}$%
\[
\delta\left(  \mathcal{U}_{1}\right)  \times\mathcal{U}_{2}=\left\{  \left(
q,\eta\right)  \in\mathcal{R}\text{ \ /\ \ }q=\delta_{v}\ \ ;\ \ v\in
\mathcal{U}_{1}\right\}  .
\]

Denote by $\delta\left(  U_{1}\right)  \times U_{2}$ the action set of all
relaxed-singular controls in $\delta\left(  \mathcal{U}_{1}\right)
\times\mathcal{U}_{2}$.

If $\left(  q,\eta\right)  \in\delta\left(  \mathcal{U}_{1}\right)
\times\mathcal{U}_{2}$, then $\left(  q,\eta\right)  =\left(  \delta_{v}%
,\eta\right)  $ with $v\in\mathcal{U}_{1}$. In this case we have for each $t$,
$\left(  q_{t},\eta_{t}\right)  =\left(  \delta_{v_{t}},\eta_{t}\right)
\in\delta\left(  U_{1}\right)  \times U_{2}$.

\ 

We equipped $\mathbb{P}\left(  U_{1}\right)  $ with the topology of stable
convergence. Since $U_{1}$ is compact, then with this topology $\mathbb{P}%
\left(  U_{1}\right)  $\ is a compact metrizable space. The stable convergence
is required for bounded measurable functions $f\left(  t,a\right)  $\ such
that for each fixed $t\in\left[  0,T\right]  $, $f\left(  t,.\right)  $\ is
continuous (Instead of functions bounded and continuous with respect to the
pair $\left(  t,a\right)  $ for the weak topology). The space $\mathbb{P}%
\left(  U_{1}\right)  $\ is equipped with its Borel $\sigma$-field, which is
the smallest $\sigma$-field such that the mapping $q\longmapsto%
{\displaystyle\int}
f\left(  s,a\right)  q\left(  ds,da\right)  $\ are measurable for any bounded
measurable function $f$, continuous with respect to $a.\ $For more details,
see Jacod and Memin $\left[  20\right]  $ and El Karoui et al $\left[
15\right]  $.

This allows us to summarize some of lemmas that we will be used in the sequel.

\ 

\begin{lemma}
(Chattering\thinspace\thinspace Lemma).\thinspace\textit{Let }$q$\textit{\ be
a predictable process with values in the space of probability measures on
}$U_{1}$\textit{. Then there exists a sequence of predictable processes
}$\left(  u^{n}\right)  _{n}$\textit{\ with values in }$U_{1}$\textit{\ such
that }%
\begin{equation}
dtq_{t}^{n}\left(  da\right)  =dt\delta_{u_{t}^{n}}\left(  da\right)
\underset{n\rightarrow\infty}{\longrightarrow}dtq_{t}\left(  da\right)  \text{
stably},\text{\textit{\ \ }}\mathcal{P}-a.s.
\end{equation}
where $\delta_{u_{t}^{n}}$ is the Dirac measure concentrated at a single point
$u_{t}^{n}$ of $U_{1}$.
\end{lemma}

\begin{proof}
See El Karoui et al $\left[  15\right]  $.
\end{proof}

\begin{lemma}
Let $q\in\mathcal{R}_{1}$ and $\left(  u^{n}\right)  _{n}\subset
\mathcal{U}_{1}$such that $\left(  29\right)  $ holds. Then for any bounded
measurable function $f:\left[  0,T\right]  \times U_{1}\rightarrow\mathbb{R}$,
such that for each fixed $t\in\left[  0,T\right]  $, $f\left(  t,.\right)  $
is continuous, we have%
\begin{equation}%
{\displaystyle\int\nolimits_{U_{1}}}
f\left(  t,a\right)  \delta_{u_{t}^{n}}\left(  da\right)  \underset
{n\rightarrow\infty}{\longrightarrow}%
{\displaystyle\int\nolimits_{U_{1}}}
f\left(  t,a\right)  q_{t}\left(  da\right)  \ ;\ dt-a.e
\end{equation}

\end{lemma}

\begin{proof}
By $\left(  29\right)  $, and the definition of the stable convergence (see
Jacod and Memin $\left[  20,\ \text{definition 1.1, page 529}\right]  $, we
have%
\[%
{\displaystyle\int\nolimits_{0}^{T}}
{\displaystyle\int\nolimits_{U_{1}}}
f\left(  t,a\right)  \delta_{u_{t}^{n}}\left(  da\right)  dt\underset
{n\rightarrow\infty}{\longrightarrow}%
{\displaystyle\int\nolimits_{0}^{T}}
{\displaystyle\int\nolimits_{U_{1}}}
f\left(  t,a\right)  q_{t}\left(  da\right)  dt.
\]

Put%
\[
g\left(  s,a\right)  =1_{\left[  0,t\right]  }\left(  s\right)  f\left(
s,a\right)  .
\]

It's clear that $g$ is bounded, measurable and continuous with respect to
$a.\ $Then, by $\left(  29\right)  $ we get%
\[%
{\displaystyle\int\nolimits_{0}^{T}}
{\displaystyle\int\nolimits_{U_{1}}}
g\left(  s,a\right)  \delta_{u_{s}^{n}}\left(  da\right)  ds\underset
{n\rightarrow\infty}{\longrightarrow}%
{\displaystyle\int\nolimits_{0}^{T}}
{\displaystyle\int\nolimits_{U_{1}}}
g\left(  s,a\right)  q_{s}\left(  da\right)  ds.
\]

By replacing $g\left(  s,a\right)  $ by its value, we have%
\[%
{\displaystyle\int\nolimits_{0}^{t}}
{\displaystyle\int\nolimits_{U_{1}}}
f\left(  s,a\right)  \delta_{u_{s}^{n}}\left(  da\right)  ds\underset
{n\rightarrow\infty}{\longrightarrow}%
{\displaystyle\int\nolimits_{0}^{t}}
{\displaystyle\int\nolimits_{U_{1}}}
f\left(  s,a\right)  q_{s}\left(  da\right)  ds.
\]

The set $\left\{  \left(  s,t\right)  \ ;\ 0\leq s\leq t\leq T\right\}  $
generate $\mathcal{B}_{\left[  0,T\right]  }$.\ Then for every $B\in
\mathcal{B}_{\left[  0,T\right]  }$, we have%
\[%
{\displaystyle\int\nolimits_{B}}
{\displaystyle\int\nolimits_{U_{1}}}
f\left(  s,a\right)  \delta_{u_{s}^{n}}\left(  da\right)  ds\underset
{n\rightarrow\infty}{\longrightarrow}%
{\displaystyle\int\nolimits_{B}}
{\displaystyle\int\nolimits_{U_{1}}}
f\left(  s,a\right)  q_{s}\left(  da\right)  ds.
\]

This implies that%
\[%
{\displaystyle\int\nolimits_{U_{1}}}
f\left(  s,a\right)  \delta_{u_{s}^{n}}\left(  da\right)  \underset
{n\rightarrow\infty}{\longrightarrow}%
{\displaystyle\int\nolimits_{U_{1}}}
f\left(  s,a\right)  q_{s}\left(  da\right)  \ ,\ \ dt-a.e.
\]

The lemma is proved.
\end{proof}

\ 

The next lemma gives the stability of the controlled SDE with respect to the
control variable.

\begin{lemma}
Let $\left(  q,\eta\right)  \in\mathcal{R}$ be a relaxed-singular control and
$x^{\left(  q,\eta\right)  }$\textit{\ the corresponding trajectory. Then
there exists a sequence }$\left(  u^{n},\eta\right)  _{n}\subset\mathcal{U}$
such that
\begin{equation}
\underset{n\rightarrow\infty}{\lim}\mathbb{E}\left[  \underset{t\in\left[
0,T\right]  }{\sup}\left\vert x_{t}^{\left(  u^{n},\eta\right)  }%
-x_{t}^{\left(  q,\eta\right)  }\right\vert ^{2}\right]  =0,
\end{equation}%
\begin{equation}
\underset{n\rightarrow\infty}{\lim}J\left(  u^{n},\eta\right)  =\mathcal{J}%
\left(  q,\eta\right)  .
\end{equation}
where $x^{\left(  u^{n},\eta\right)  }$ denotes the solution of equation
$\left(  1\right)  $ associated with $\left(  u^{n},\eta\right)  .$
\end{lemma}

\begin{proof}
i) Proof of\textbf{\ }$\left(  31\right)  $. We have%
\begin{align*}
\mathbb{E}\left\vert x_{t}^{\left(  u^{n},\eta\right)  }-x_{t}^{\left(
q,\eta\right)  }\right\vert ^{2}  &  \leq C%
{\displaystyle\int\nolimits_{0}^{t}}
\mathbb{E}\left\vert b\left(  s,x_{s}^{\left(  u^{n},\eta\right)  },u_{s}%
^{n}\right)  -b\left(  s,x_{s}^{\left(  q,\eta\right)  },u_{s}^{n}\right)
\right\vert ^{2}ds\\
&  +C%
{\displaystyle\int\nolimits_{0}^{t}}
\mathbb{E}\left\vert b\left(  s,x_{s}^{\left(  q,\eta\right)  },u_{s}%
^{n}\right)  -%
{\displaystyle\int\nolimits_{U_{1}}}
b\left(  s,x_{s}^{\left(  q,\eta\right)  },a\right)  q_{s}\left(  da\right)
\right\vert ^{2}ds\\
&  +C%
{\displaystyle\int\nolimits_{0}^{t}}
\mathbb{E}\left\vert \sigma\left(  s,x_{s}^{\left(  u^{n},\eta\right)  }%
,u_{s}^{n}\right)  -\sigma\left(  s,x_{s}^{\left(  q,\eta\right)  },u_{s}%
^{n}\right)  \right\vert ^{2}ds\\
&  +C%
{\displaystyle\int\nolimits_{0}^{t}}
\mathbb{E}\left\vert \sigma\left(  s,x_{s}^{\left(  q,\eta\right)  },u_{s}%
^{n}\right)  -%
{\displaystyle\int\nolimits_{U_{1}}}
\sigma\left(  s,x_{s}^{\left(  q,\eta\right)  },a\right)  q_{s}\left(
da\right)  \right\vert ^{2}ds
\end{align*}

Since $b$ and $\sigma$ are uniformly Lipschitz with respect to $x$, then%
\begin{align}
&  \mathbb{E}\left\vert x_{t}^{\left(  u^{n},\eta\right)  }-x_{t}^{\left(
q,\eta\right)  }\right\vert ^{2}\\
&  \leq C%
{\displaystyle\int\nolimits_{0}^{t}}
\mathbb{E}\left\vert x_{s}^{\left(  u^{n},\eta\right)  }-x_{s}^{\left(
q,\eta\right)  }\right\vert ^{2}ds\nonumber\\
&  +C%
{\displaystyle\int\nolimits_{0}^{t}}
\mathbb{E}\left\vert
{\displaystyle\int\nolimits_{U_{1}}}
b\left(  s,x_{s}^{\left(  q,\eta\right)  },a\right)  \delta_{u_{s}^{n}}\left(
da\right)  -%
{\displaystyle\int\nolimits_{U_{1}}}
b\left(  s,x_{s}^{\left(  q,\eta\right)  },a\right)  q_{s}\left(  da\right)
\right\vert ^{2}ds\nonumber\\
&  +C%
{\displaystyle\int\nolimits_{0}^{t}}
\mathbb{E}\left\vert
{\displaystyle\int\nolimits_{U_{1}}}
\sigma\left(  s,x_{s}^{\left(  q,\eta\right)  },a\right)  \delta_{u_{s}^{n}%
}\left(  da\right)  -%
{\displaystyle\int\nolimits_{U_{1}}}
\sigma\left(  s,x_{s}^{\left(  q,\eta\right)  },a\right)  q_{s}\left(
da\right)  \right\vert ^{2}ds\nonumber
\end{align}

Since $b$ and $\sigma$ are bounded, measurable and continuous with respect to
$a$, then by $\left(  30\right)  $ and the dominated convergence theorem, the
second and third terms in the right hand side of the above inequality tend to
zero as $n$ tends to infinity. We conclude then by using Gronwall's lemma and
Bukholder-Davis-Gundy inequality.

\ 

ii) Proof of $\left(  32\right)  $. By using the Cauchy-Schwartz inequality
and the fact that $g$ and $h$ are uniformly Lipschitz with respect to $x$, we
get%
\begin{align*}
&  \left\vert J\left(  q^{n},\eta\right)  -\mathcal{J}\left(  q,\eta\right)
\right\vert \\
&  \leq C\left(  \mathbb{E}\left\vert x_{T}^{\left(  u^{n},\eta\right)
}-x_{T}^{\left(  q,\eta\right)  }\right\vert ^{2}\right)  ^{1/2}+C\left(
\int_{0}^{T}\mathbb{E}\left\vert x_{t}^{\left(  u^{n},\eta\right)  }%
-x_{t}^{\left(  q,\eta\right)  }\right\vert ^{2}ds\right)  ^{1/2}\\
&  +\left(  \mathbb{E}\int_{0}^{T}\left\vert
{\displaystyle\int\nolimits_{U_{1}}}
h\left(  s,x_{s}^{\left(  q,\eta\right)  },a\right)  \delta_{u_{s}^{n}}\left(
da\right)  dt-%
{\displaystyle\int\nolimits_{U_{1}}}
h\left(  t,x_{t}^{\left(  q,\eta\right)  },a\right)  q_{t}\left(  da\right)
\right\vert ^{2}dt\right)  ^{1/2}.
\end{align*}

By $\left(  31\right)  $, the first and second terms in the right hand side
converge to zero. Moreover, since $h$ is bounded, measurable and continuous in
$a$, then by $\left(  30\right)  $ and the dominated convergence theorem, the
third term in the right hand side tends to zero as $n$ tends to infinity. This
prove $\left(  32\right)  $.
\end{proof}

\begin{lemma}
As a consequence of $\left(  32\right)  $, the strict-singular and the
relaxed-singular control problems have the same value functions. That is
\begin{equation}
\underset{\left(  v,\eta\right)  \in\mathcal{U}}{\inf}J\left(  v,\eta\right)
=\underset{\left(  q,\eta\right)  \in\mathcal{R}}{\inf}\mathcal{J}\left(
q,\eta\right)  .
\end{equation}

\end{lemma}

\begin{proof}
Let $\left(  u,\xi\right)  \in\mathcal{U}$ and $\left(  \mu,\xi\right)
\in\mathcal{R}$ be respectively a strict-singular and relaxed-singular
controls such that%
\begin{align}
J\left(  u,\xi\right)   &  =\underset{\left(  v,\eta\right)  \in\mathcal{U}%
}{\inf}J\left(  v,\eta\right)  ,\\
\mathcal{J}\left(  \mu,\xi\right)   &  =\underset{\left(  q,\eta\right)
\in\mathcal{R}}{\inf}\mathcal{J}\left(  q,\eta\right)  .
\end{align}

By $\left(  36\right)  $, we have
\[
\mathcal{J}\left(  \mu,\xi\right)  \leq\mathcal{J}\left(  q,\eta\right)
\text{, }\forall\left(  q,\eta\right)  \in\mathcal{R}\text{.}%
\]

Since $\delta\left(  \mathcal{U}_{1}\right)  \times\mathcal{U}_{2}%
\subset\mathcal{R}$, then%
\[
\mathcal{J}\left(  \mu,\xi\right)  \leq\mathcal{J}\left(  q,\eta\right)
\text{, }\forall\left(  q,\eta\right)  \in\delta\left(  \mathcal{U}%
_{1}\right)  \times\mathcal{U}_{2}\text{.}%
\]

Since $\left(  q,\eta\right)  \in\delta\left(  \mathcal{U}_{1}\right)
\times\mathcal{U}_{2}$, then $\left(  q,\eta\right)  =\left(  \delta_{v}%
,\eta\right)  $, where $v\in\mathcal{U}_{1}$. Then we get%
\[
\left\{
\begin{array}
[c]{c}%
x^{\left(  q,\eta\right)  }=x^{\left(  v,\eta\right)  },\\
\mathcal{J}\left(  q,\eta\right)  =J\left(  v,\eta\right)  .
\end{array}
\right.
\]

Hence%
\[
\mathcal{J}\left(  \mu,\xi\right)  \leq J\left(  v,\eta\right)  \text{,
}\forall\left(  v,\eta\right)  \in\mathcal{U}\text{.}%
\]

The control $\left(  u,\xi\right)  $ becomes an element of $\mathcal{U}$, then
we get%
\begin{equation}
\mathcal{J}\left(  \mu,\xi\right)  \leq J\left(  u,\xi\right)  \text{.}%
\end{equation}

On the other hand, by $\left(  35\right)  $ we have%
\begin{equation}
J\left(  u,\xi\right)  \leq J\left(  v,\eta\right)  \text{, }\forall\left(
v,\eta\right)  \in\mathcal{U}\text{.}%
\end{equation}

The process $\mu$ becomes an element of $\mathcal{R}_{1}$, then by the
Chattering lemma (Lemma $12$), there exists a sequence $\left(  u^{n}\right)
_{n}\subset\mathcal{U}_{1}$ such that%
\[
dt\mu_{t}^{n}\left(  da\right)  =dt\delta_{u_{t}^{n}}\left(  da\right)
\underset{n\rightarrow\infty}{\longrightarrow}dt\mu_{t}\left(  da\right)
\text{ stably},\text{\textit{\ \ }}\mathcal{P}-a.s.
\]

Relation $\left(  38\right)  $ holds for every $\left(  v,\eta\right)
\in\mathcal{U}$. This is true for $\left(  u^{n},\xi\right)  \in
\mathcal{U},\ \forall n\in\mathbb{N}.$ We get then%
\[
J\left(  u,\xi\right)  \leq J\left(  u^{n},\xi\right)  \text{, }\forall
n\in\mathbb{N}\text{,}%
\]

By using $\left(  32\right)  $ and letting $n$ go to infinity in\ the above
inequality, we get%
\begin{equation}
J\left(  u,\xi\right)  \leq\mathcal{J}\left(  \mu,\xi\right)  .
\end{equation}

Finally, by $\left(  37\right)  $ and $\left(  39\right)  $, we have
\[
J\left(  u,\xi\right)  =\mathcal{J}\left(  \mu,\xi\right)  .
\]

The lemma is proved.
\end{proof}

\ 

To establish necessary optimality conditions for strict-singular controls, we
need the following lemma

\begin{lemma}
The strict-singular control $\left(  u,\xi\right)  $ minimizes $J$ over
$\mathcal{U}$ if and only if the relaxed-singular control $\left(  \mu
,\xi\right)  =\left(  \delta_{u},\xi\right)  $ minimizes $\mathcal{J}$ over
$\mathcal{R}$.
\end{lemma}

\begin{proof}
Suppose that $\left(  u,\xi\right)  $ minimizes the cost $J$ over $\mathcal{U}
$, then
\[
J\left(  u,\xi\right)  =\underset{\left(  v,\eta\right)  \in\mathcal{U}}{\inf
}J\left(  v,\eta\right)  \text{.}%
\]

By using $\left(  34\right)  $, we get%
\[
J\left(  u,\xi\right)  =\underset{\left(  q,\eta\right)  \in\mathcal{R}}{\inf
}\mathcal{J}\left(  q,\eta\right)  \text{.}%
\]

Since $\left(  \mu,\xi\right)  =\left(  \delta_{u},\xi\right)  $, then%
\begin{equation}
\left\{
\begin{array}
[c]{c}%
x^{\left(  \mu,\xi\right)  }=x^{\left(  u,\xi\right)  },\\
\mathcal{J}\left(  \mu,\xi\right)  =J\left(  u,\xi\right)  ,
\end{array}
\right.
\end{equation}

This implies that%
\[
\mathcal{J}\left(  \mu,\xi\right)  =\underset{\left(  q,\eta\right)
\in\mathcal{R}}{\inf}\mathcal{J}\left(  q,\eta\right)  .
\]

Conversely, if $\left(  \mu,\xi\right)  =\left(  \delta_{u},\xi\right)  $
minimize $\mathcal{J}$ over $\mathcal{R}$, then%
\[
\mathcal{J}\left(  \mu,\xi\right)  =\underset{\left(  q,\eta\right)
\in\mathcal{R}}{\inf}\mathcal{J}\left(  q,\eta\right)  .
\]

From $\left(  34\right)  $, we get%
\[
\mathcal{J}\left(  \mu,\xi\right)  =\underset{\left(  v,\eta\right)
\in\mathcal{U}}{\inf}J\left(  v,\eta\right)  .
\]

Since $\left(  \mu,\xi\right)  =\left(  \delta_{u},\xi\right)  $, then
relations $\left(  40\right)  $ hold, and we obtain%
\[
J\left(  u,\xi\right)  =\underset{\left(  v,\eta\right)  \in\mathcal{U}}{\inf
}J\left(  v,\eta\right)  .
\]

The proof is completed.
\end{proof}

\ 

The following lemma, who will be used to establish sufficient optimality
conditions for strict-singular controls, shows that we get the results of the
above lemma if we replace $\mathcal{R}$ by $\delta\left(  \mathcal{U}%
_{1}\right)  \times\mathcal{U}_{2}.$

\begin{lemma}
The strict-singular control $\left(  u,\xi\right)  $ minimizes $J$ over
$\mathcal{U}$ if and only if the relaxed control $\left(  \mu,\xi\right)
=\left(  \delta_{u},\xi\right)  $ minimizes $\mathcal{J}$ over $\delta\left(
\mathcal{U}_{1}\right)  \times\mathcal{U}_{2}$.
\end{lemma}

\begin{proof}
Let $\left(  \mu,\xi\right)  =\left(  \delta_{u},\xi\right)  $ be an optimal
relaxed-singular control minimizing the cost $\mathcal{J}$ over $\delta\left(
\mathcal{U}_{1}\right)  \times\mathcal{U}_{2}$, we have then%
\begin{equation}
\mathcal{J}\left(  \mu,\xi\right)  \leq\mathcal{J}\left(  q,\eta\right)
\text{,\ \ }\forall\left(  q,\eta\right)  \in\delta\left(  \mathcal{U}%
_{1}\right)  \times\mathcal{U}_{2}.
\end{equation}

Since $q\in\delta\left(  \mathcal{U}_{1}\right)  $, then there exists
$v\in\mathcal{U}_{1}$ such that $q=\delta_{v}.$ Hence, $\left(  \delta
_{v},\eta\right)  =\left(  q,\eta\right)  $, and since $\left(  \mu
,\xi\right)  =\left(  \delta_{u},\xi\right)  $, it is easy to see that%
\begin{equation}
\left\{
\begin{array}
[c]{c}%
x^{\left(  \mu,\xi\right)  }=x^{\left(  u,\xi\right)  },\\
x^{\left(  q,\eta\right)  }=x^{\left(  v,\eta\right)  },\\
\mathcal{J}\left(  \mu,\xi\right)  =J\left(  u,\xi\right)  ,\\
\mathcal{J}\left(  q,\eta\right)  =J\left(  v,\eta\right)  .
\end{array}
\right.
\end{equation}

By $\left(  41\right)  $ and $\left(  42\right)  $, we get then%
\[
J\left(  u,\xi\right)  \leq J\left(  v,\eta\right)  ,\ \ \forall\left(
v,\eta\right)  \in\mathcal{U}\text{.}%
\]

Conversely, let $\left(  u,\xi\right)  $ be a strict-singular control
minimizing the cost $J$ over $\mathcal{U}$. Then%
\[
J\left(  u,\xi\right)  \leq J\left(  v,\eta\right)  ,\ \ \forall\left(
v,\eta\right)  \in\mathcal{U}\text{.}%
\]

Since the controls $u,v$ $\in\mathcal{U}_{1}$, then there exist $\mu
,q\in\delta\left(  \mathcal{U}_{1}\right)  $ such that $\mu=\delta_{u}$ and
$q=\delta_{v}$. Then%
\begin{align*}
\left(  \mu,\xi\right)   &  =\left(  \delta_{u},\xi\right)  ,\\
\left(  q,\eta\right)   &  =\left(  \delta_{v},\eta\right)  .
\end{align*}

This implies that relations $\left(  42\right)  $ hold. Consequently, we get%
\[
\mathcal{J}\left(  \mu,\xi\right)  \leq\mathcal{J}\left(  q,\eta\right)
\text{,\ \ }\forall\left(  q,\eta\right)  \in\delta\left(  \mathcal{U}%
_{1}\right)  \times\mathcal{U}_{2}.
\]

The proof is completed.
\end{proof}

\subsection{Necessary optimality conditions for strict-singular controls}

\ 

Define the Hamiltonian in the strict case from $\left[  0,T\right]
\times\mathbb{R}^{n}\times U_{1}\times\mathbb{R}^{n}\times\mathcal{M}_{n\times
d}\left(  \mathbb{R}\right)  $ into $\mathbb{R}$ by%
\[
H\left(  t,x,v,p,P\right)  =h\left(  t,x,v\right)  +b\left(  t,x,v\right)
p+\sigma\left(  t,x,v\right)  P.
\]

\begin{theorem}
(Necessary optimality conditions for strict-singular controls in global form).
Suppose that $\left(  u,\xi\right)  $\textit{\ is an optimal strict-singular
control minimizing the cost }$J$\textit{\ over }$\mathcal{U}$\textit{\ and
}$x^{\left(  u,\xi\right)  }$\textit{\ denotes the solution of }$\left(
1\right)  $\textit{\ controlled by }$\left(  u,\xi\right)  $\textit{. Then,
there exists an unique pair of adapted processes}
\[
\left(  p^{\left(  u,\xi\right)  },P^{\left(  u,\xi\right)  }\right)
\in\mathcal{L}^{2}\left(  \left[  0,T\right]  ;\mathbb{R}^{n}\right)
\times\mathcal{L}^{2}\left(  \left[  0,T\right]  ;\mathbb{R}^{n\times
d}\right)  ,
\]
\textit{solution of the backward SDE}%
\begin{equation}
\left\{
\begin{array}
[c]{l}%
dp_{t}^{\left(  u,\xi\right)  }=-H_{x}\left(  t,x_{t}^{\left(  u,\xi\right)
},u_{t},p_{t}^{\left(  u,\xi\right)  },P_{t}^{\left(  u,\xi\right)  }\right)
dt+P_{t}^{\left(  u,\xi\right)  }dW_{t},\\
p_{T}^{\left(  u,\xi\right)  }=g_{x}(x_{T}^{\left(  u,\xi\right)  }),
\end{array}
\right.
\end{equation}
such that%
\begin{equation}
H\left(  t,x_{t}^{\left(  u,\xi\right)  },u_{t},p_{t}^{\left(  u,\xi\right)
},P_{t}^{\left(  u,\xi\right)  }\right)  =\underset{v_{t}\in U_{1}}{\inf
}H\left(  t,x_{t}^{\left(  u,\xi\right)  },v_{t},p_{t}^{\left(  u,\xi\right)
},P_{t}^{\left(  u,\xi\right)  }\right)  ,\ a.e,\ a.s.
\end{equation}%
\begin{equation}
\mathcal{P}\left\{  \forall t\in\left[  0,T\right]  ,\;\forall i\;;\;\left(
k_{_{i}}\left(  t\right)  +G_{i}^{\ast}\left(  t\right)  .p_{t}^{\left(
u,\xi\right)  }\right)  \geq0\right\}  =1,
\end{equation}%
\begin{equation}
\mathcal{P}\left\{  \sum_{i=1}^{d}\mathbf{1}_{\left\{  k_{_{i}}\left(
t\right)  +G_{i}^{\ast}\left(  t\right)  p_{t}^{\left(  u,\xi\right)  }%
\geq0\right\}  }d\xi_{t}^{i}=0\right\}  =1.
\end{equation}

\end{theorem}

\begin{proof}
The optimal strict-singular control $\left(  u,\xi\right)  $ is an element of
$\mathcal{U}$, then there exists $\left(  \mu,\xi\right)  \in\delta\left(
\mathcal{U}_{1}\right)  \times\mathcal{U}_{2}$ such that
\[
\left(  \mu,\xi\right)  =\left(  \delta_{u},\xi\right)  .
\]

Since $\left(  u,\xi\right)  $ minimizes the cost $J$ over $\mathcal{U}$, then
by lemma $16$, $\left(  \mu,\xi\right)  $ minimizes $\mathcal{J}$ over
$\mathcal{R}$. Hence, by the necessary optimality conditions for
relaxed-singular controls (Theorem $10$), there exists an unique pair of
adapted processes $\left(  p^{\left(  \mu,\xi\right)  },P^{\left(  \mu
,\xi\right)  }\right)  $, solution of $\left(  21\right)  $, such that
relations $\left(  23\right)  ,\ \left(  24\right)  $ and $\left(  25\right)
$ hold.

Since $\delta\left(  U_{1}\right)  \subset\mathbb{P}\left(  U_{1}\right)  $,
then \ by $\left(  23\right)  $, we get%
\begin{equation}
\mathcal{H}\left(  t,x_{t}^{\left(  \mu,\xi\right)  },\mu_{t},p_{t}^{\left(
\mu,\xi\right)  },P_{t}^{\left(  \mu,\xi\right)  }\right)  \leq\mathcal{H}%
\left(  t,x_{t}^{\left(  \mu,\xi\right)  },q_{t},p_{t}^{\left(  \mu
,\xi\right)  },P_{t}^{\left(  \mu,\xi\right)  }\right)  ,\ \forall q_{t}%
\in\delta\left(  U_{1}\right)  ,\ a.e,\ a.s.
\end{equation}

Since $q\in\delta\left(  \mathcal{U}_{1}\right)  $, then there exists
$v\in\mathcal{U}_{1}$ such that $q=\delta_{v}$.

We note that $v$ is an arbitrary element of $\mathcal{U}_{1}$ since $q$ is arbitrary.

Now, since $\left(  \mu,\xi\right)  =\left(  \delta_{u},\xi\right)  $ and
$\left(  q,\xi\right)  =\left(  \delta_{v},\xi\right)  $, we can easily see
that%
\begin{equation}
\left\{
\begin{array}
[c]{c}%
x^{\left(  \mu,\xi\right)  }=x^{\left(  u,\xi\right)  },\\
x^{\left(  q,\xi\right)  }=x^{\left(  v,\xi\right)  },\\
\left(  p^{\left(  \mu,\xi\right)  },P^{\left(  \mu,\xi\right)  }\right)
=\left(  p^{\left(  u,\xi\right)  },P^{\left(  u,\xi\right)  }\right)  ,\\
\mathcal{H}\left(  t,x_{t}^{\left(  \mu,\xi\right)  },\mu_{t},p_{t}^{\left(
\mu,\xi\right)  },P_{t}^{\left(  \mu,\xi\right)  }\right)  =H\left(
t,x_{t}^{\left(  u,\xi\right)  },u_{t},p_{t}^{\left(  u,\xi\right)  }%
,P_{t}^{\left(  u,\xi\right)  }\right)  ,\\
\mathcal{H}\left(  t,x_{t}^{\left(  \mu,\xi\right)  },q_{t},p_{t}^{\left(
\mu,\xi\right)  },P_{t}^{\left(  \mu,\xi\right)  }\right)  =H\left(
t,x_{t}^{\left(  u,\xi\right)  },v_{t},p_{t}^{\left(  u,\xi\right)  }%
,P_{t}^{\left(  u,\xi\right)  }\right)  ,
\end{array}
\right.
\end{equation}
where $\left(  p^{\left(  u,\xi\right)  },P^{\left(  u,\xi\right)  }\right)  $
is the unique solution of $\left(  43\right)  $.

By using $\left(  47\right)  $ and $\left(  48\right)  $, we deduce $\left(
44\right)  $. Relations $\left(  45\right)  $ and $\left(  46\right)  $
follows immediately from $\left(  24\right)  ,\left(  25\right)  $ and
$\left(  48\right)  .$ The proof is completed.
\end{proof}

\begin{remark}
Bahlali and Mezerdi $\left[  2\right]  $, established necessary optimality
conditions for strict-singular controls of the second-order with two adjoint
processes. The result of the above theorem improves that of $\left[  2\right]
$, in the sense where, we consider the same strict-singular control problem,
with nonconvex control domain and a general state equation in which the
control variable enters both the drift and the diffusion coefficients, and we
establish necessary optimality conditions of the first-order with only one
adjoint process.
\end{remark}

\subsection{Sufficient optimality conditions for strict-singular controls}

\ 

\begin{theorem}
(Sufficient optimality conditions for strict-singular controls). Assume that
the functions $g$\ and $x\longmapsto H\left(  t,x,q,p,P\right)  $\ are convex.
Then, $\left(  u,\xi\right)  $\ is an optimal solution of problem $\left\{
\left(  1\right)  ,\left(  2\right)  ,\left(  3\right)  \right\}  $ if it
satisfies $\left(  44\right)  ,\ \left(  45\right)  $ and $\left(  46\right)
.$
\end{theorem}

\begin{proof}
Let $\left(  u,\xi\right)  \in\mathcal{U}$\ be a strict-singular control
(candidate to be optimal) and $\left(  v,\eta\right)  $ an arbitrary element
of $\mathcal{U}$.

The controls $u,v$\ are elements of $\mathcal{U}_{1}$, then there exist
$\mu,q\in\delta\left(  \mathcal{U}_{1}\right)  $\ such that $\mu=\delta_{u}$
and $q=\delta_{v}$. Hence,
\begin{align*}
\left(  \mu,\xi\right)   &  =\left(  \delta_{u},\xi\right)  ,\\
\left(  q,\eta\right)   &  =\left(  \delta_{v},\eta\right)  .
\end{align*}

This implies that relations $\left(  48\right)  $ hold. Then, by $\left(
44\right)  ,\ \left(  45\right)  $ and $\left(  46\right)  $, we deduce
respectively the relaxed relations
\begin{equation}
\mathcal{H}\left(  t,x_{t}^{\left(  \mu,\xi\right)  },\mu_{t},p_{t}^{\left(
\mu,\xi\right)  },P_{t}^{\left(  \mu,\xi\right)  }\right)  =\underset{q_{t}%
\in\mathbb{\delta}\left(  U_{1}\right)  }{\inf}\mathcal{H}\left(
t,x_{t}^{\left(  \mu,\xi\right)  },q_{t},p_{t}^{\left(  \mu,\xi\right)
},P_{t}^{\left(  \mu,\xi\right)  }\right)  ,\ a.e,\ a.s,
\end{equation}%
\begin{equation}
\mathcal{P}\left\{  \forall t\in\left[  0,T\right]  ,\;\forall i\;;\;\left(
k_{_{i}}\left(  t\right)  +G_{i}^{\ast}\left(  t\right)  .p_{t}^{\left(
\mu,\xi\right)  }\right)  \geq0\right\}  =1,
\end{equation}%
\begin{equation}
\mathcal{P}\left\{  \sum_{i=1}^{d}\mathbf{1}_{\left\{  k_{_{i}}\left(
t\right)  +G_{i}^{\ast}\left(  t\right)  p_{t}^{\left(  \mu,\xi\right)  }%
\geq0\right\}  }d\xi_{t}^{i}=0\right\}  =1.
\end{equation}

We remark that the infimum in $\left(  49\right)  $ is taken over
$\mathbb{\delta}\left(  U_{1}\right)  .$

Now, since $H$\ is convex in $x$, it is easy to see that $\mathcal{H}$\ is
convex in $x$, and since $g$\ is convex, then by using $\left(  49\right)
,\ \left(  50\right)  $ and $\left(  51\right)  $, and by the same proof that
in theorem $11$, we show that $\left(  \mu,\xi\right)  $ minimizes the cost
$\mathcal{J}$ over $\delta\left(  \mathcal{U}_{1}\right)  \times
\mathcal{U}_{2}$. Then, by Lemma $17$, we deduce that $\left(  u,\xi\right)  $
minimizes the cost $J$\ over $\mathcal{U}$. The theorem is proved.
\end{proof}

\begin{remark}
The sufficient optimality conditions for strict-singular controls are proved
without assuming neither the convexity of $U_{1}$ nor that of the Hamiltonian
$H$ in $v$.
\end{remark}

\end{document}